\documentclass[leqno]{amsart}
\usepackage[utf8]{inputenc}
\usepackage[english]{babel}
\usepackage{newclude}
\usepackage{amsfonts,amssymb,amsmath,amsgen, amsthm}
\usepackage{color, xcolor}
\usepackage{pdfsync}
\usepackage{enumerate}
\usepackage{centernot}
\usepackage{comment}

\newtheorem{theorem}{Theorem}[section]
\newtheorem{lem}[theorem]{Lemma}
\newtheorem{prop}[theorem]{Proposition}
\newtheorem{cor}[theorem]{Corollary}
\newtheorem{remark}[theorem]{Remark}
\theoremstyle{definition}
\newtheorem{defin}[theorem]{Definition}

\def\R{\mathbb R}

\def\N{\mathbb N}
\def\d{\partial}

\title[Constrained nonlinear parabolic equations]{Existence and Asymptotic Behavior for $L^2$-norm Preserving Nonlinear Heat Equations}

\author[P. Antonelli]{Paolo Antonelli}
\address{Gran Sasso Science Institute, viale Francesco Crispi, 7, 67100 L'Aquila, Italy}
\email{paolo.antonelli@gssi.it}

\author[P. Cannarsa]{Piermarco Cannarsa}
\address{University of Rome “Tor Vergata”, Rome, Italy}
\email{cannarsa@mat.uniroma2.it}

\author[B. Shakarov]{Boris Shakarov}
\address{Gran Sasso Science Institute, viale Francesco Crispi, 7, 67100 L'Aquila, Italy}
\email{boris.shakarov@gssi.it}

\date{\today}

\begin{document}
	\maketitle
	\begin{abstract}
We consider a nonlinear parabolic equation with a nonlocal term which preserves the $L^2$-norm of the solution. 
We study the local and global well-posedness on a bounded domain, as well as the whole Euclidean space, in $H^1$. 
Then we study the asymptotic behavior of solutions. In general, we obtain weak convergence in $H^1$ to a stationary state. For  a ball, we prove strong  convergence to the ground state when the initial condition is positive.
	\end{abstract}

\section{Introduction}\label{sec:heat_intro}

In this work, we study the existence, uniqueness, and asymptotic behavior of the solution to the following nonlinear, nonlocal parabolic equation
	\begin{equation}\label{eq:heat}
		\begin{cases}
			& \partial_t u = \Delta u + g |u|^{2\sigma} u + \mu[u] u, \\
			& u_{|\partial\Omega} = 0,\\
			& u(0) = u_0 \in H^1_0(\Omega),
		\end{cases}
	\end{equation}	
where $u:\R_+\times\Omega\to\R$, $g\in\R$, $\sigma>0$, and the functional $\mu[\cdot]$ is defined by
\begin{equation}\label{eq:mu}
	\mu[u] := \frac{\| \nabla u \|^2_{L^2} - g \| u\|^{2\sigma + 2}_{L^{2\sigma + 2}}}{{\|u\|_{L^2}^2}}.
\end{equation}
Here we consider both the case when $\Omega\subset\R^d$ is a regular, bounded domain with boundary $\partial\Omega$ of class $C^2$, and when $\Omega = \R^d$ is the Euclidean space. In the latter case, the Dirichlet boundary condition may be interpreted as
\begin{equation*}
u\to0, \quad\textrm{as}\quad |x|\to\infty.
\end{equation*}
Let us note that the functional $\mu[u]$ ensures that the $L^2$-norm is preserved along the flow, namely 
\begin{equation*}
	\| u(t)\|_{L^2} = \| u_0\|_{L^2},
\end{equation*}
for any $t>0$.
Hence, $\mu[u]$ may be interpreted as a Lagrange multiplier that takes into account the fact that the solution is constrained to stay on a given sphere in $L^2$, whose radius is prescribed by the initial datum. In this respect, equation \eqref{eq:heat} may be viewed as an $L^2$ gradient flow constrained on a manifold. More precisely, having set $\lambda=\|u_0\|_{L^2}$, let us define
\begin{equation*}
\mathcal M=\{u\in H^1_0(\Omega)\,:\,\|u\|_{L^2}=\lambda\}.
\end{equation*}
Then \eqref{eq:heat} can be written as
\begin{equation*}
\d_tu=-\nabla_{\mathcal M}E(u),
\end{equation*}
where $\nabla_{\mathcal M}$ is the $L^2$-gradient projected onto the tangent space $T_u\mathcal M$ and the energy functional is defined by
\begin{equation}\label{eq:en2} 
	E[u(t)] := \frac{1}{2}\|\nabla u(t)\|^2_{L^2} - \frac{g}{2\sigma + 2 } \| u(t)\|^{2\sigma + 2}_{L^{2\sigma + 2}}.
\end{equation}
Nonlocal heat flows arise in geometrical problems where some $L^p$-norms, related to specific geometrical quantities, are required to be preserved by the dynamics, see for instance \cite{Au98}, \cite{St00} and the references therein.
In particular, the Yamabe flow was set up to show the existence of a solution to the celebrated Yamabe problem. More precisely, the conformal transformation bringing a metric $g_0$, with scalar curvature $R_0$, into a metric with constant scalar curvature can be determined by the asymptotic limit for large times of the evolutionary problem
\begin{equation*}
\d_tu^p=s(t)u^p+\frac{4(d-1)}{(d-2)}\Delta_0u-R_0u,
\end{equation*}
where $\Delta_0$ is the Laplace-Beltrami operator and $p=2^*-1$. Here the Lagrange multiplier $s(t)$ is determined in such a way that the volume, proportional to $\int|u|^{2^*}$, is preserved. We refer the interested reader to \cite{Ye94, Br05} for more details.
\newline
Our study is also motivated by recent numerical works that exploit normalized gradient flow methods to compute the ground states for some models in Bose-Einstein condensation, see \cite{ChSuTo00, Du02, BaDu04} for instance. 
In this perspective, our results may be seen as a rigorous justification for the methods proposed in those papers, see Theorem \ref{thm:asy} below.
\newline
Equation \eqref{eq:heat} with $g=0$ was already rigorously studied in \cite{CaLi09}. Apart from the nonlocal term, equation \eqref{eq:heat} becomes linear when $g=0$. This fact allows for a detailed description of the asymptotic behavior of the solution, which is characterized by the initial condition. In \cite{CaLi09}, this result is also used to study a singularly perturbed heat flow, with applications to an optimal partition problem \cite{CaLi07}, see also \cite{Du08} for a numerical implementation of the method.
\newline
In \cite{MaCh09}, the authors consider equation \eqref{eq:heat} on a closed manifold (i.e., a compact manifold with no boundary) and $g<0$. They show the global existence of solutions and study their asymptotic behavior. Let us remark that this choice of $g$ makes the energy functional \eqref{eq:en2} always non-negative definite.
\newline	
In this work, we extend the above results by considering $g \in \R$ and the case when either $\Omega$ is a bounded domain or $\Omega=\R^d$. 
The positive sign of $g$ implies that the energy does not control the $H^1$-norm of the solution in general. 
In particular, we show the existence of global-in-time solutions in the case $g>0$ and $0<\sigma<2/d$, for arbitrarily large initial data. Let us remark that, under the same assumptions $g>0$ and $0<\sigma<2/d$, the classical semilinear parabolic equation, i.e. equation \eqref{eq:heat} with $\mu=0$, experiences a possible finite-time blow-up, see \cite[Theorem 17.6]{QuSo19} for instance. The $L^2$ constraint provided by the dynamics \eqref{eq:heat} then prevents this kind of singularity formation. 
\newline
Moreover, in the case of $\sigma\in\left(\frac{2}{d}, \frac{2}{(d-2)^+}\right)$ and $g\geq0$, it is possible to exploit some arguments borrowed from the standard potential well method (see \cite{HoRo07,PaSa75} for instance) to show global existence of solutions for some initial data. More precisely, we determine a subset of $H^1$ that is invariant for the dynamics and initial data belonging to that region emanating global solutions. On the other hand, we also identify a set of initial data whose evolution experiences a grow-up behavior, see Theorem \ref{thm:growup} below.
\newline	
Our second main goal is to investigate the asymptotic behavior for large times. As in \cite{MaCh09}, by standard compactness methods it is possible to show that along sequences of times going to infinity, the solution is converging to a stationary solution. On the other hand, when $\Omega$ is a bounded domain we can further improve this result by showing that solutions emanating from non-negative initial data converge to the ground state solution, see Theorem \ref{thm:asy} for a more precise statement.
\newline	
We now present our main results. First of all, we prove the local well-posedness of the Cauchy problem \eqref{eq:heat}. Let us remark that the usual fixed point argument, see for instance \cite[Sect. 16]{QuSo19}, cannot be applied to \eqref{eq:heat} in a straightforward way. Indeed, there is a quite delicate interplay between the power-type nonlinearity and the nonlocal term. A general proof of the local well-posedness result appears to be missing so far. For this reason, in proving Theorem \ref{thm:lwp_intro} we are going to adopt different strategies. We refer to the beginning of Section \ref{sec:exis} for a more detailed discussion.
\newline
 We introduce the following conditions on $\sigma$
 \begin{equation}\label{eq:sigmaCon}
 0 < \sigma < \frac{1}{(d-2)^+} \quad \mbox{or} \quad \frac{1}{2} \leq \sigma < \frac{2}{(d-2)^+}.
 \end{equation}

\begin{theorem}\label{thm:lwp_intro}
Let $\Omega$ be a bounded domain with $C^2$ boundary and $0 < \sigma < \frac{2}{(d-2)^+}$ or $\Omega = \R^d$ and $\sigma$ verifying the conditions in \eqref{eq:sigmaCon}. Then for any
		$ u_0 \in H^1_0(\Omega)$ there exists a maximal time of existence $T_{max} = T_{max}(\| u_0\|_{H^1}) >0$ and a unique local solution $u \in C([0,T_{max}), H^1_0(\Omega))$ to \eqref{eq:heat}. Moreover, either $T_{max} = \infty$ or $T_{max} < \infty$ and
 		\begin{equation*}
		 \lim_{t \rightarrow T_{max}} \| \nabla u(t)\|_{L^2} = \infty.
		\end{equation*}
\end{theorem}

As previously mentioned, in the case of $g>0$, the energy functional in \eqref{eq:en2} is indefinite, hence global well-posedness does not follow straightforwardly from available a priori bounds. In the following theorem, we present some conditions under which the solutions constructed above can be extended globally.

\begin{theorem}\label{thm:gwp_intro}
	Let the assumptions of Theorem \ref{thm:lwp_intro} hold. If we further assume $g \leq 0$ or $g \geq 0$ and $\sigma < \frac{\sigma}{2}$, then for any
	$ u_0 \in H^1_0(\Omega)$, there exists a unique global solution $u \in C([0,\infty), H^1_0(\Omega))$ to \eqref{eq:heat}.
\end{theorem}
	
The next theorem will provide additional conditions under which the solution is global. To state it, let us define 
\begin{equation*}
	\mathcal W = \left\{ u \in H^1_0(\Omega)\, ;\, E[u] < p, \, I[u] > 0 \right\},
\end{equation*}
where 
\begin{equation*}
I[u] = \| \nabla u\|_{L^2}^2 - g\| u\|_{L^{2\sigma + 2}}^{2\sigma+2}, \quad 	p = \inf \left\{E[u]\, ; \, u \in H^1_0(\Omega) \setminus \{ 0\}, I[u] = 0 \right\},
\end{equation*}
and $\Omega$ is a bounded domain. Since $\Omega$ is a bounded domain, $\mathcal W$ is a nonempty set, see Section \ref{sec:potwel} and \ref{ss:gwp} for more details. In the spirit of the potential well method, see \cite{PaSa75} for instance, we have the following global existence result.

\begin{theorem}\label{thm:pot1}
	Assume $0 < \sigma < \frac{2}{(d-2)^+}$ and $g > 0$. If $u_0 \in \mathcal W$, then the corresponding solution $u$ to \eqref{eq:heat} is global in time.
\end{theorem}

The potential well method can also be used in the whole Euclidean space $\R^d$ to obtain additional sufficient conditions for global existence, as in \cite{HoRo07}. To state the next theorem, we introduce the following notations. Let $\frac{2}{d} < \sigma < \frac{2}{(d-2)^+}$ and let $Q$ be the unique positive solution to
\begin{equation*}
 \Delta Q - Q + |Q|^{2\sigma}Q = 0.
\end{equation*}
Let $\mathcal K$ be defined as 
\begin{equation*}
	\mathcal K = \{ f \in H^1(\R^d): \, E[f] \| f\|_{L^2}^{2\alpha} < E[Q] \| Q\|_{L^2}^{2\alpha}, \, \|\nabla f\|_{L^2}^2 \| f\|_{L^2}^{2\alpha} < \|\nabla Q\|_{L^2}^2 \| Q\|_{L^2}^{2\alpha} \},
\end{equation*}
where
\begin{equation*}
	\alpha = \frac{2 - (d - 2)\sigma}{d\sigma - 2}.
\end{equation*}
Once again, we note that the above set is nonempty (see Section \ref{ss:gwp}).
Then the following holds.

\begin{theorem}\label{thm:potK}
 Assume $\frac{2}{d} < \sigma < \frac{2}{(d-2)^+}$ and $g > 0$. If $u_0 \in \mathcal K$, then the corresponding solution $u$ to \eqref{eq:heat} is global in time.
\end{theorem}

While it is interesting to further extend the range of models for which global well-posedness holds, another main open problem is to study the possible occurrence of a finite-time blow-up in some specific regimes.
Let us recall that the usual arguments for nonlinear parabolic equations (see for instance \cite[Theorem 17.1]{QuSo19}) do not apply here, due to the conservation of the $L^2$-norm. 
On the other hand, our next result shows a grow-up scenario for some data, in the case of $g \geq 0$ and $\sigma \geq \frac{2}{d}$. Notice that such conditions complement the hypothesis of Theorem~\ref{thm:gwp_intro} where global existence is proven for any data.
\begin{theorem}\label{thm:growup}
	 Let $g \geq 0 $ and $\sigma \geq \frac{2}{d}$. Let $u_0 \in H^1(\R^d)$ be such that $E[u_0] < 0$. Let $T_{max} (u_0) > 0$ be the maximum time of existence of the corresponding strong solution to \eqref{eq:heat}. Then either $T_{max}(u_0) < \infty$ (and $\lim_{t \rightarrow T_{max}} \| \nabla u(t)\|_{L^2} = \infty$) or  $T_{max}(u_0) = \infty$  and
	 \begin{equation*}
	 \limsup_{t \rightarrow \infty} \| \nabla u(t)\|_{L^2} = \infty.
	 \end{equation*}
	\end{theorem}
After having shown the existence of global solutions, we now turn our attention to their asymptotic behavior.
When $g=0$ in \eqref{eq:heat} and $\Omega$ is bounded, the eigenspaces associated with the Dirichlet Laplacian are invariant along the dynamics. As a byproduct, the solution asymptotically approaches the eigenspace of least energy, which contains a non-trivial component of the initial datum, see \cite{CaLi09} where this is proved by explicit calculations. 
On the other hand, allowing for $g\neq0$ leads to more complex dynamics. 
\newline
In \cite{MaCh09}, this problem is addressed in the case of $g<0$ for bounded domains. The authors prove convergence to some stationary state (no uniqueness is given) for some sequence of times going to infinity.
\newline
The next proposition extends this result to arbitrary $g\neq0$ for both $\Omega$ bounded and $\Omega=\R^d$.

\begin{prop}
Let $u \in C([0,\infty), H^1_0(\Omega))$ be the global solution constructed in Theorem \ref{thm:gwp_intro}, with initial datum $u(0)=u_0\in H^1_0(\Omega)$.
Then there exists a sequence $\{ t_n \}_{n \in \N}$, $t_n \rightarrow \infty$ as $n \rightarrow \infty$, such that 
\begin{equation*}
u(t_n) \rightharpoonup u_\infty \ \mbox{ in } \ H^1_0(\Omega) \ \ \mbox{and} \ \ \mu[u(t_n)] \rightarrow \mu[u_\infty] \ \mbox{ as } \ n\rightarrow \infty,
\end{equation*}
where $u_\infty$ solves the stationary equation
\begin{equation*}
\Delta u_\infty + g|u_\infty|^{2\sigma} u_\infty + \mu[u_\infty] u_\infty = 0
\end{equation*}
and $\| u_\infty\|_{L^2} \leq \| u_0\|_{L^2}$.
\end{prop}

Let us emphasize that, in the previous proposition, we only obtain the upper bound $\| u_\infty\|_{L^2} \leq \|u_0\|_{L^2}$ instead of the equality, because of a possible loss of mass at infinity (when $\Omega=\R^d$). This result may be improved when considering a bounded domain. By exploiting further compactness properties, available in this case, it is possible to show the strong convergence of the sequence.
\newline
Moreover, in the case when $g\leq0$, $0<\sigma<\frac{2}{(d-2)^+}$ or $g>0$, $0<\sigma<2/d$, it is well known (see Section \ref{ss:stst}) that there exist regular domains $\Omega \subset \R^d$ for which there exists a unique, positive, radially symmetric stationary solution, that is also a minimizer of the energy functional \eqref{eq:en2} under a constraint on the total mass:
\begin{equation}\label{eq:minpr}
		e_0= \inf \left\{E[u]\,: u \in H^1_0(\Omega), \, \| u\|_{L^2} = \|u_0\|_{L^2} \right\}.
 	\end{equation}
 Such a solution is usually called the ground state solution. Notice that a minimizer of the problem \eqref{eq:minpr} satisfies the elliptic equation
 	\begin{equation}\label{eq:st_st1}
		0 = \Delta Q + gQ^{2\sigma+1} + \mu[Q] Q,
	\end{equation}
 where $\mu[Q]$ is a Lagrange multiplier. 
 The maximum principle then ensures that global solutions emanated by non-negative initial data converge asymptotically to the ground state.

\begin{theorem}\label{thm:asy}
		Let $g \leq 0$, $\sigma < \frac{2}{(d-2)^+}$ or $g > 0$, $\sigma < \frac{2}{d}$. Let $u_0 \in H^1_0(\Omega)$, $u_0 \geq 0$, $u_0 \centernot{\equiv} 0$ and let $u \in C([0,\infty), H^1_0(\Omega))$ be the solution to \eqref{eq:heat}. Let $\Omega \subset \R^d$ be a bounded domain such that there exists a unique positive solution $Q \in H^1_0(\Omega)$ to equation \eqref{eq:st_st1} which is also the unique minimizer of the problem \eqref{eq:minpr}. Then 
		\begin{equation*}
			u(t) \rightarrow Q \ \mbox{ in } \ H^1_0(\Omega)\ \mbox{ as } \ t\rightarrow \infty.
		\end{equation*}
	\end{theorem}
	
The problem we address in this work can also be looked at as an example of a bilinear control system. Indeed, consider the nonlinear heat equation
\begin{equation}\label{eq:bilinear}
    \partial_t u = \Delta u + g |u|^{2\sigma} u + p(t)q(x) u 
\end{equation}
where $q$ is a given smooth function defined on $\Omega$ and $p\in L^2(0,T)$ is a scalar function of our choice. The problem consists of finding, for any initial condition $u_0$, a control $p$ which steers the solution to a given target at time $T$. In \cite{BeMa20}, this question was addressed for $d=1$, $g = 0$, with the ground state solution as a target, proving an exact controllability result in any time $T > 0$. This result was later extended, in \cite{AlCaUr22}, to arbitrary eigensolutions of evolution equations of parabolic type in one space dimension, under various boundary conditions. Notice that, in both \cite{BeMa20} and \cite{AlCaUr22}, $q$ is not allowed to be the identity. 
\newline
On the other hand, the approach proposed in \cite{CaLi09} applies to the above problem in arbitrary dimension, with $q(x)\equiv 1$ and $g=0$, by choosing $p(t)=\mu[u(t)]$. In \cite{CaLi09}, however, the target is attained asymptotically (as $T\to\infty$) also by an appropriate selection of the initial condition. So, in this work, we extend the approach by \cite{CaLi09} to nonlinear flows $(g\neq 0)$ restricting the target to the ground state (for $u_0\geq 0$).
Let us mention that in \cite{BeDuCo20}, the authors studied the well-posedness of \eqref{eq:heat} via the fixed point argument, similarly to Theorem \ref{thm:lwp} below. They also showed that solutions starting close to a local minimum of the energy at fixed mass converge to it asymptotically. 
\newline
Finally, let us also mention that a similar analysis can be also performed on other models, such as the Navier-Stokes equations, see for instance \cite{CaPuRo09, brzezniak20182d} where the two-dimensional system is studied under various constraints. 
\newline
This work is organized as follows: in Section \ref{sec:prel}, we will present some useful preliminary results regarding the heat semigroup generated by the Laplacian on $L^2(\Omega)$, the stationary states of system \eqref{eq:heat}, and the potential well method. In Section \ref{sec:exis}, we will prove our main results (Theorem\ref{thm:lwp_intro} up to \ref{thm:potK}) regarding the local and global existence of solutions to the system \eqref{eq:heat}. In Section \ref{sec:asy}, we will study the asymptotic behavior and give the proof of Theorem \ref{thm:asy} and Theorem \ref{thm:growup}.

\section{Preliminaries}\label{sec:prel}

\subsection{On the notion of solutions}

In this subsection, we introduce the different notions of solutions that are going to be used in this work. Here, $\Omega$ can be either a bounded domain with $C^2$ boundary or the Euclidean space $\Omega = \R^d$.

\begin{defin} 
Let $u_0 \in H^1_0(\Omega).$ We give the following three definitions.
\begin{enumerate}[a)]
 \item {\bf Strong solution:} We define $u$ to be a strong solution if there exists $T>0$ such that $u\in H^1([0,T],L^2(\Omega)) \cap L^2([0,T],{H^2}(\Omega)) \cap C([0,T],H^1_0(\Omega)) $ satisfies the equation \eqref{eq:heat} a.e. in $\Omega$ for any $t\in [0,T]$. 
 \item {\bf Mild solution:} 
 We define $u$ to be a mild solution if there exists $T>0$ such that $u\in C([0,T],H^1_0(\Omega))$ and for any $t \in [0,T]$, $u$ satisfies the integral formulation 
\begin{equation*}
 u(t) = e^{t\Delta} u_0 + \int_0^t e^{(t-s)\Delta} (g |u(s)|^{2\sigma} u(s) + \mu[u(s)]u(s)) ds,
\end{equation*}
where $e^{t\Delta}$ is the Dirichlet heat semigroup defined in Section \ref{sec:semig}.
\item {\bf Weak solution:} We define $u$ to be a weak solution if there exists $T>0$ such that $u\in L^2([0,T],H^1_0(\Omega))$ and for any $t \in [0,T]$, $u$ satisfies 
\begin{equation*}
 \frac{d}{dt} (\phi, u(t)) = (\Delta \phi, u) + (\phi, g|u|^{2\sigma} u + \mu[u] u), 
\end{equation*}
for any $\phi \in H^2(\bar{\Omega}) \cap H^1_0(\Omega)$ and $(\phi, u(t)) $ is absolutely continuous in time.
\end{enumerate}
\end{defin}

The following result states that any mild solution is a strong solution, see \cite[Example 51.10, Example 51.28, Lemma 17.5]{QuSo19}.

\begin{prop}\label{prop:upgrade}
 Let $\sigma < \frac{2}{(d - 2)^+}$, $u_0 \in H^1_0(\Omega)$ and $u\in C([0,T),H^1_0(\Omega)) $ be the corresponding mild solution to \eqref{eq:heat}. Then the energy 
 \begin{equation*}
 E[u(t)] = \int_\Omega \frac{1}{2} |\nabla u(t)|^2 - \frac{g}{2\sigma + 2} |u|^{2\sigma + 2} dx,
 \end{equation*}
 is so that $E \in C([0,T)) \cap C^1((0,T))$ as a function of time. Moreover, $$u \in C([0,T), H^1_0(\Omega)) \cap C^1((0,T), L^2(\Omega)) \cap C((0,T), H^2(\Omega)),$$ and
 \begin{equation}\label{eq:derEn}
 E[u(t)] - E[u_0] = - \int_0^t \int_\Omega |\partial_s u(s)|^2 dxds.
 \end{equation}
\end{prop}

\subsection{Heat semigroup}\label{sec:semig}

% In this subsection, we collect some useful properties of the Dirichlet heat semigroup. We refer the reader to \cite{QuSo19} for more details. \newline	

% Let $\Omega$ be a bounded domain in $\R^d$ with $C^2$ boundary. Let $-A$ denote the Dirichlet Laplacian in $L^2(\Omega)$, that is the Laplacian on $L^2(\Omega)$ subject to homogeneous Dirichlet boundary conditions. Then $-A$ is a non-negative self-adjoint operator and it generates a $C^0$-semigroup $e^{-tA}$ on $L^2(\Omega)$.
In this work, we will denote with $e^{t\Delta}$ both the semigroup generated by the Laplacian 
% $e^{-tA}$ 
on $L^2(\Omega)$ and $L^2(\R^d)$. 
% Note that 
The function $u = e^{t\Delta} u_0$ solves the linear heat equation
\begin{equation}\label{eq:linheat}
\begin{cases}
			& \partial_t u = \Delta u, \ \mbox{in} \ \ (0,\infty) \times \Omega, \\
			& u(0) = u_0 \in H^1_0(\Omega).
		\end{cases}
\end{equation} 
with Dirichlet boundary conditions $u_{|\partial\Omega} = 0,$ if $\Omega$ is bounded. 
% The Dirichlet heat kernel
% % .It 
% is a positive, $C^\infty$ function $G_\Omega:\Omega \times \Omega \times (0,\infty) \to \R$ such that 
% \begin{equation*}
% e^{t\Delta}f = \int_\Omega G_\Omega (x,y,t) f(y) dy,
% \end{equation*}
% for any $f \in L^p(\Omega)$, $1 \leq p \leq \infty$. 
% In addition, $ G_{\Omega_1}(x,y,t) \leq G_{\Omega_2} (x,y,t)$ whenever $\Omega_1 \subset \Omega_2$, $x,y \in \Omega_1$. %Moreover, 
% If $\Omega = \R^d$, then
% \begin{equation}\label{eq:gauss}
% 		G_{\R^d} (x,y,t) = \frac{1}{(4\pi t)^\frac{d}{2}} e^{- \frac{|x - y|^2}{4t}}, \quad x,y \in \R^d, \ t \in (0,\infty),
% 	\end{equation}
% % is the Gaussian heat kernel. Let us observe that this implies the following inequality 
% and thus
% \begin{equation*}
%  0 < G_{\Omega} (x,y,t) \leq \frac{1}{(4\pi t)^\frac{d}{2}} e^{- \frac{|x - y|^2}{4t}},
% \end{equation*}
% for any $\Omega \subset \R^d$, $x,y \in \Omega$ and $\ t \in (0,\infty)$.
% In view of the discussion above,
One can prove the following $L^p-L^q$ smoothing estimates of the heat semigroup (see \cite[Proposition 48.4]{QuSo19}):
\begin{prop}\label{prp:lplq}
 Let $\{e^{t\Delta}\}_{t >0}$ be the heat semigroup in $\R^d$ or the Dirichlet heat semigroup in $\Omega$. Then for any $1 \leq p < q \leq \infty$, $t > 0$ and $f \in L^p(\Omega)$, 
 \begin{equation*}
 \left\| e^{t\Delta} f \right\|_{L^q} \leq \frac{1}{(4\pi t)^{\frac{d}{2} \left(\frac 1 p - \frac 1 q\right)}} \| f\|_{L^p}.
 \end{equation*}
\end{prop}

As a consequence of the $L^p-L^q$ smoothing property in Proposition \ref{prp:lplq}, we also obtain the following space-time estimates. 

\begin{prop}\label{prp:sptime}
 Let $\{e^{t\Delta}\}_{t >0}$ be the heat semigroup in $\R^d$ or the Dirichlet heat semigroup in $\Omega$. Let $2 \leq q,r \leq \infty$ satisfy
 \begin{equation}\label{eq:qrcon}
 \frac 2 q + \frac d r = \frac d 2
 \end{equation}
 with $(q,r,d) \neq (2,\infty,2)$.
 Then for any $f \in L^2$, we have
 \begin{equation}\label{eq:sptime1}
 \left\| e^{t\Delta} f \right\|_{L^q((0,\infty),L^r)} \lesssim \| f\|_{L^2}.
 \end{equation}
 Moreover, if $(\rho,\gamma)$ is another pair satisfying condition \eqref{eq:qrcon}, then 
 \begin{equation}\label{eq:sptime2}
 \left\| \int_0^te^{(t - s) \Delta} f(s) \ ds \right\|_{L^q((0,\infty),L^r)} \lesssim \| f\|_{L^{\rho'}((0,\infty),L^{\gamma'})},
 \end{equation}
 for any $f \in L^{\rho'}((0,\infty),L^{\gamma'})$.
\end{prop}
In the proposition above and in what follows we use the following notation: for any $p \geq 1 $, we denote by $p' \in \R$ the constant such that 
\begin{equation*}
 1 = \frac{1}{p} + \frac{1}{p'},
\end{equation*}
with the convention that $p' = \infty$ if $p =1$. \newline 
% The next proposition gives a useful smoothing estimate for the gradient. The case $\Omega = \R^d$ follows directly from calculations involving the Gaussian heat kernel \eqref{eq:gauss}. For a general bounded domain $\Omega$, 
The following estimate holds, see \cite[Theorem 16.3]{LaSoUr88}.
\begin{prop} \label{prop:smooth_par}
 Let $\Omega$ be a regular domain of class $C^2$ boundary. Let $\{e^{t\Delta}\}_{t >0}$ be the heat semigroup in $\R^d$ or the Dirichlet heat semigroup in $\Omega$. Then there exists a constant $C >0$ such that for any $F \in L^2([0,\infty), L^2(\Omega))$
 \begin{equation} \label{eq:smooth}
 \left\|\nabla \int_0^te^{(t-s)\Delta}F(s)\,ds\right\|_{L^\infty([0,\infty),L^2)}\leq C\|F\|_{L^2([0,\infty),L^2)}.
 \end{equation}
\end{prop}

Finally we also notice that for any $u_0 \in H^1$, a solution $u(t)$ to the linear system \eqref{eq:linheat} dissipates the $L^2$-norm in the following way
\begin{equation}\label{eq:linmass}
 \| e^{t\Delta} u_0 \|_{L^2}^2 = \| u_0\|_{L^2}^2 - 2\int_0^t \| \nabla e^{s \Delta} u_0\|_{L^2}^2 \, ds.
\end{equation}

\subsection{Stationary states}\label{ss:stst}

In this subsection, we recall some results concerning the stationary solutions of equation \eqref{eq:heat}, namely satisfying
\begin{equation}\label{eq:st_st}
	0 = \Delta Q + gQ^{2\sigma+1} + \mu[Q] Q.
\end{equation}
Note that, for any given $\alpha >0$, if the problem
\begin{equation}\label{eq:min_pr}
			e_\alpha = \inf \left\{E[u]\,: u \in H^1_0(\Omega), \, \| u\|_{L^2} = \alpha \right\},
\end{equation} 
admits a minimizer $Q \in H^1_0(\Omega)$, then $Q$ solves equation \eqref{eq:st_st} where the constant $\mu[Q]$ is a Lagrange multiplier. Moreover, we recall the following Pohozaev's identities.

\begin{lem}
		 Let $Q \in H^1(\R^d)$ be a solution to \eqref{eq:st_st}. Then
		 \begin{equation}\label{eq:poh1}
		 \mu[Q] \| Q\|_{L^2}^2 = \| \nabla Q\|_{L^2}^2 - g \| Q\|_{L^{2\sigma + 2}}^{2\sigma +2},
		 \end{equation}
		 \begin{equation}\label{eq:poh2}
		 \frac{d-2}{2} \| \nabla Q\|_{L^2} = \frac{\mu[Q]d}{2} \| Q\|_{L^2}^2 + \frac{dg}{2\sigma + 2} \| Q\|_{L^{2\sigma + 2}}^{2\sigma +2}.
		 \end{equation}
\end{lem}

\begin{proof}
 We take the scalar product of \eqref{eq:st_st} with $Q$ and obtain \eqref{eq:poh1}. Similarly, we obtain \eqref{eq:poh2} by taking the scalar product of \eqref{eq:st_st} with $x \cdot \nabla Q$. 
\end{proof}	

As a consequence of \eqref{eq:poh1}, we have that
\begin{equation*}
	\mu[Q] = \frac{\| \nabla Q \|^2_{L^2} - g \| Q\|^{2\sigma + 2}_{L^{2\sigma + 2}}}{{\|Q\|_{L^2}^2}}.
\end{equation*}
	In what follows, we denote by $B(0,R) = \{x \in \R^d\, ;\, |x| < R\}$ a ball in $\R^d$, centered in the origin. 
	We start with the following standard proposition (see \cite{Li84_2, Sh17, St82}).
	\begin{prop} \label{prp:gs}
		Let $g > 0$ and $\sigma < \frac{2}{d}$. Let $\Omega= B(0,R)$ or $\Omega = \R^d$. Then for any $\alpha >0$, the minimization problem \eqref{eq:min_pr}
		admits a real non-negative, radially symmetric minimizer.
	\end{prop}
	We proceed by recalling the result of Gidas, Ni, Nirenberg \cite[Theorem $1$]{GiNiNi79} stating that any real non-negative solution to \eqref{eq:st_st} with $\Omega = B(0, R)$ or $\Omega = \R^d$ is radially symmetric.
	\begin{prop}\label{prp:gs1}
		Let $Q \geq 0$ be a real, $C^2(\bar{\Omega})$ solution to \eqref{eq:st_st} with $Q_{|\partial \Omega} = 0$ and let $\Omega = B(0,R)$ or $\Omega = \R^d$. Then $Q$ is radially symmetric, strictly positive, and strictly decreasing in $|x|$. 
	\end{prop}
	Next, we recall the result stating the uniqueness of positive, radial solutions due to Kwong \cite{Kw89} and Mcleod-Serrin \cite[Theorem $1$]{McSe87}.
	\begin{prop}\label{prp:gs2}
		For any $d > 1$, there exists exactly one positive solution $Q \in H^1_{rad}(\Omega)$ of the problem
		\begin{equation*}
			0 = Q'' + \frac{d-1}{r} Q' + |Q|^{2\sigma} Q + \mu[Q] Q,
		\end{equation*}
		where $\Omega=B(0,R)$, $\Omega = \{ x \in \R^d, \, a \leq |x| \leq b\}$ or $\Omega = \R^d$.
	\end{prop}
	Let us notice that Propositions \ref{prp:gs}, \ref{prp:gs1}, \ref{prp:gs2} imply the following.
	\begin{cor}\label{cor:uniq}
		Let $g >0$ and $\sigma < \frac{2}{d}$. Let $\Omega = B(0,R)$. Then for any $\alpha >0$, the unique, positive minimizer to the problem
		\begin{equation*}
			e_\alpha = \inf \left\{E[u]\,: u \in H^1_0(\Omega), \, \| u\|_{L^2} = \alpha \right\},
		\end{equation*}
		is also the unique, strictly positive solution to equation \eqref{eq:st_st}.
	\end{cor}

\subsection{Potential well in a bounded domain}\label{sec:potwel}

In this subsection, we present some preliminary definitions and results regarding the potential well method in a bounded set. We use this method to provide an alternative sufficient condition for global existence for the Cauchy problem \eqref{eq:heat}, stated in Theorem \ref{thm:pot1}. This argument was first introduced in \cite{Ts72} and \cite{PaSa75} for parabolic problems. We refer the interested reader to \cite[Section 19]{QuSo19} for more details.
 
Let $\Omega$ be a bounded domain, $0 < \sigma < \frac{2}{d-2}$ and $g >0$. We fix $g =1$ without losing generality to facilitate the exposition. 
We define the functional $I$ by
\begin{equation*}
	I[u] = \| \nabla u\|_{L^2}^2 - \| u\|_{L^{2\sigma + 2}}^{2\sigma + 2}, \quad u \in H^1_0(\Omega).
\end{equation*}
Notice that $I$ can also be written as 
\begin{equation}\label{eq:itoE}
 I[u] = 2E[u] - \frac{\sigma}{\sigma + 1} \| u\|_{L^{2\sigma + 2}}^{2\sigma + 2},
\end{equation}
where $E$ is the energy functional defined in \eqref{eq:en2}. 
\newline
The potential well associated with problem \eqref{eq:heat} is the set 
\begin{equation}\label{eq:wdef}
	\mathcal W = \left\{ u \in H^1_0(\Omega)\, ;\, E[u] < p, \, I[u] > 0 \right\} \cup \{0\},
\end{equation}
where $p$ is defined by 
\begin{equation}\label{eq:infP}
	p = \inf \left\{E[u]\, ; \, u \in H^1_0(\Omega) \setminus \{ 0\}, I[u] = 0 \right\}. 
\end{equation}
We shall see in Lemma \ref{lem:p} below that
\begin{equation}\label{eq:plamd}
	p = \frac{\sigma}{2\sigma + 2} \Lambda^\frac{2\sigma + 2}{\sigma},
\end{equation}
where $\Lambda = \Lambda(\sigma, \Omega)$ is the best constant in the Sobolev embedding $H^1_0(\Omega) \hookrightarrow L^{2\sigma + 2}(\Omega)$, i.e,
\begin{equation}\label{eq:sobcon}
	\Lambda = \inf \left\{ \frac{\|\nabla u\|_{L^2}}{\|u\|_{L^{2\sigma + 2}}}\, ;\, u \in H^2(\Omega) \cap H^1_0(\Omega) , u\neq 0 \right\}.
\end{equation}
We also define the exterior of the potential well as 
\begin{equation}\label{eq:zdef}
	\mathcal Z = \left\{ u \in H^1_0(\Omega)\, ;\, E[u] < p, \, I[u] < 0 \right\}.
\end{equation}

\begin{lem}\label{lem:p}
	Let $\Omega$ be bounded and $0 < \sigma \leq \frac{2}{(d-2)^+}$. Then property \eqref{eq:plamd} is true. Moreover, if $0 < \sigma < \frac{2}{(d-2)^+}$ then the infimum in \eqref{eq:infP} is achieved. 
\end{lem}

\begin{proof}
	Let $0 < \sigma < \frac{2}{(d-2)^+}$. Then the infimum in \eqref{eq:sobcon} is attained for some $v \in H^1_0(\Omega)$, due to the compactness of the embedding $H^1_0(\Omega) \hookrightarrow L^{2\sigma + 2}(\Omega)$.
 By multiplying $v$ for a suitable constant $\nu$, we may suppose that $I[v] = 0$. Therefore
 \begin{equation*}
	 \| \nabla v\|_{L^2} = \| u\|_{L^{2\sigma + 2}}^{2\sigma + 2} = \Lambda^{-(2\sigma + 2)} \|\nabla v\|_{L^2}^{\sigma + 1}.
 \end{equation*}
 This implies that
 \begin{equation*}
	 E[v] = \frac{\sigma}{(2\sigma + 2)} \Lambda^\frac{2\sigma + 2}{\sigma}.
 \end{equation*}
 Thus, there exists an element $v$ so that $E[v] = p$ where $p$ is defined as in \eqref{eq:plamd} and $I[v] = 0$. To show that such a $p$ is the infimum of the problem \eqref{eq:infP} we notice that for any $u\in H^1_0(\Omega)$, $u \neq v$, $I[u] = 0$, the inequality
 \begin{equation*}
		\| \nabla u\|_{L^2} = \| u\|_{L^{2\sigma + 2}}^{2\sigma + 2} \leq \Lambda^{-(2\sigma + 2)} \|\nabla u\|_{L^2}^{\sigma + 1}.
 \end{equation*}
 implies that
 \begin{equation*}
	 E[u] \geq \frac{\sigma}{(2\sigma + 2)} \Lambda^\frac{2\sigma + 2}{\sigma}.
 \end{equation*}
\end{proof}

Next, we state a sufficient smallness condition for a function in $H^1_0(\Omega)$ to belong to $\mathcal{ W}$.

\begin{prop}\label{prp:minim}
 Let $f \in H^1_0(\Omega)$	 be so that $\| \nabla f \|_{L^2} < \sqrt{2p}$. Then $f \in \mathcal W$.
\end{prop}

\begin{proof}
 Notice that the condition $\| \nabla f \|_{L^2} < \sqrt{2p}$ implies that 
	\begin{equation*}
	 E[f] < p - \frac{1}{2\sigma + 2} \| f\|_{L^{2\sigma}}^{2\sigma + 2} < p.
	\end{equation*}
 On the other hand, using \eqref{eq:plamd},\eqref{eq:sobcon} and 
	\begin{equation*}
	 		\| \nabla f \|_{L^2} < \sqrt{2p} < \Lambda^{\frac{\sigma +1}{\sigma}},
	 \end{equation*} 
	 we obtain
	 \begin{equation*}
	 		\| f \|_{L^{2\sigma + 2}}^{2\sigma + 2} \leq \Lambda^{-(2\sigma + 2)} \left( \int_\Omega |\nabla f|^2\right)^{\sigma + 1} < \| \nabla f\|_{L^2}^2.
	 \end{equation*}
	 This implies that $I[f] > 0$ and consequently that $f \in \mathcal{W}$. 
\end{proof}

 \subsection{Potential well method in $\R^d$.}\label{sec:potK}

 In this subsection, we present some preliminary results regarding the potential well method in the whole Euclidean space $\R^d$. This argument was introduced in \cite{HoRo07} to prove a sufficient condition for the global existence of the Nonlinear Schr\"odinger equation. 
 \newline
We suppose that $\frac{2}{d} < \sigma < \frac{2}{(d-2)^+}$ and $g>0$. We fix $g = 1$ without losing generality. We recall that the best constant in the Gagliardo-Nirenberg inequality
\begin{equation}\label{eq:GN}
	\| f \|_{L^{2\sigma + 2}}^{2\sigma + 2} \leq C_{GN} \| \nabla f\|_{L^2}^{d\sigma} \| f\|_{L^2}^{2 + 2\sigma - d\sigma}
\end{equation}
is given by 
\begin{equation*}
C_{GN} = W(Q) = \frac{\| Q\|_{L^{2\sigma + 2}}^{2\sigma + 2}}{ \| \nabla Q\|_{L^2}^{d\sigma} \| Q\|_{L^2}^{2 + 2\sigma - d\sigma}}
\end{equation*}
where $Q$ is the unique positive solution of the elliptic equation 
\begin{equation*}
	-\Delta Q + Q - Q^{2\sigma + 1} = 0. 
\end{equation*}
In \cite{We82}, this profile is found as the maximizer of the following problem
\begin{equation*}
	W_{max} = \sup \{ W(f)\, : \, f \in H^1(\R^d), \, f \neq 0\}.
\end{equation*}
We define the potential well as the set 
\begin{equation*}
	\mathcal K = \{ f \in H^1(\R^d): \, E[f] \| f\|_{L^2}^{2\alpha} < E[Q] \| Q\|_{L^2}^{2\alpha}, \ \|\nabla f\|_{L^2}^2 \| f\|_{L^2}^{2\alpha} < \|\nabla Q\|_{L^2}^2 \| Q\|_{L^2}^{2\alpha} \},
\end{equation*}
where
\begin{equation*}
	\alpha = \frac{2 - (d - 2)\sigma}{d\sigma - 2}
\end{equation*}
and $Q$ is the ground state defined above. 
\newline
Finally, we observe that 
	\begin{equation}\label{eq:cGN}
		C_{GN} = \frac{2\sigma + 2}{d\sigma} \left( \| \nabla Q\|_{L^2}\| Q\|_{L^2}^\alpha \right)^{- d\sigma -2}.
	\end{equation}
Indeed, from the Pohozaev's identities \eqref{eq:poh1}, \eqref{eq:poh2} it follows that
 	\begin{equation*}
 		\| Q\|_{L^2}^2 = \frac{2 - (d - 2)\sigma}{d\sigma} \| \nabla Q\|_{L^2}^2 = \frac{2 - (d - 2)\sigma}{2\sigma + 2} \| \nabla Q\|_{L^{2\sigma + 2}}^{2\sigma + 2} = \frac{4 - 2(d - 2)\sigma}{d\sigma - 2} E[Q]. 
 	\end{equation*}
	This implies \eqref{eq:cGN}.

\section{Existence of Solutions} \label{sec:exis}

In this section, we are going to prove our local and global well-posedness results.
As already mentioned in the Introduction, there is a delicate interplay between (the mathematical difficulties determined by) the power-type nonlinearity and the nonlocal term involving $\mu[u]$. More precisely, it does not seem possible to prove Theorem \ref{thm:lwp_intro} by a general contraction argument. Indeed, the presence of the nonlocal term would require setting up the fixed point in Sobolev spaces, say $L^\infty_tH^1_0$ for instance. On the other hand, the power-type nonlinearity is not locally Lipschitz in Sobolev spaces for low values of $\sigma$, say $0<\sigma<\frac12$, see Remark \ref{rem:lip}. 
For this reason, we are going to provide three different proofs of the local well-posedness.
\newline
% \textcolor{red}{
% The first one, stated in Theorem \ref{thm:proof2} below, constructs a sequence of approximating solutions through an iterative procedure. The solution to the original problem is then found by a compactness argument. For this reason, the proof is well suited when dealing with a bounded domain, but it fails in the case $\Omega=\R^d$, where, in general, the $L^2$ constraint may not be satisfied. \newline
% }
% \textcolor{blue}{
The first one, stated in Theorem \ref{thm:proof2} below, involves a two step procedure. First, we find a solution to our problem for any data in $H^2(\Omega) \cap H^1_0(\Omega) $. Then, by a density argument, we show the our problem is locally well posed in $H^1_0(\Omega)$. This proof is well suited when dealing with a bounded domain due to the compact embedding $H^1(\Omega) \hookrightarrow L^2(\Omega)$.
\newline
The second proof, stated in Theorem \ref{thm:proof1}, is based on a fixed point argument and crucially relies on smoothing properties of the heat semi-group, see Proposition \ref{prop:smooth_par}. This proof requires the nonlinearity to satisfy the condition $\sigma<\frac{1}{(d-2)^+}$. 
\newline
The third proof relies on the classical fixed point argument. To overcome the difficult interplay between the nonlocal term and the power-type nonlinearity, a lower bound on $\sigma$ is required. In particular, for $\sigma \geq \frac{1}{2}$, the power-type nonlinearity is locally Lipschitz in Sobolev spaces and we perform a contraction argument using the natural and stronger distance. We remark that unifying the results of the second and third proof, we cover all the energy subcritical range $0 < \sigma < \frac{2}{(d-2)^+}$ for $d \leq 4$. 
\newline
Finally in subsection \ref{ss:gwp} we globally extend the previous results, under the assumptions on the nonlinear term stated in Theorem \ref{thm:gwp_intro} above. Then we provide additional sufficient conditions on the initial data for global existence, as stated in Theorem \ref{thm:pot1} and Theorem \ref{thm:potK}, using the potential well method.

\subsection{Local existence via the Schauder fixed point theorem.} 

In this subsection, we will prove the local well-posedness for bounded domains and the entire energy subcritical case $0 < \sigma < \frac{2}{(d - 2)^+}$.

\begin{theorem}\label{thm:proof2}
		Let $0<\sigma < \frac{2}{(d-2)^+}$, $\Omega\subset\R^d$ be a regular bounded domain. For any $u_0 \in H^1_0(\Omega)$, there exists a maximal time of existence $T_{max}>0$ and a unique solution $u \in C([0,T_{max}), H^1_0(\Omega))$ to \eqref{eq:heat}. Moreover, either $T_{max} = \infty$, or $T_{max} < \infty$ and
		\begin{equation}\label{eq:bu1}
			\lim_{t \rightarrow T_{max}} \| \nabla u(t)\|_{L^2} = \infty.
		\end{equation}
\end{theorem}

The proof aims at avoiding the use of the fixed point argument on the integral formulation of \eqref{eq:heat}. Due to the interplay between the nonlinearities, this standard argument  does not work for $\sigma < \frac{1}{2}$, see Remark \ref{rem:lip} below. 

 In particular, the steps of this proof are as follows.
\begin{enumerate}
    \item We show the existence of solutions for the model where the nonlinearity $\mu[u]$ is substituted by  a function $\lambda \in L^\infty[0,\infty)$ and gather estimates on the $H^1$ norm of this model. (Lemma \ref{lem:fixed} and Corollary \ref{prp:it_time}).
    \item We employ the Schauder fixed point argument. Here we need extra regularity of the initial data $u_0 \in H^2 \cap H^1_0$ to obtain the compactness of the image. Moreover, we have to substitute the $L^2$-norm in the denominator of $\mu$ by a constant to have a convex domain of the map. (Proposition \ref{prpSchauder}, equation \eqref{eqMuAlpha}).
    \item We use a density argument to show the existence of solutions for $u_0 \in H^1_0$. (Theorem \ref{thmLocExisDom})
    \item Finally, we prove the equivalence of the problem \eqref{eq:heat} and \eqref{eqMuAlpha} where the constant at step $2)$ is chosen to be the square of the $L^2$ norm of the initial condition. (Proposition \ref{prpEquivalence}).
    \item We show the uniqueness of the solution in Proposition \ref{prp:uni}.
\end{enumerate}

 The compact embedding $H^1(\Omega) \hookrightarrow L^2(\Omega)$ will be crucial to guarantee the compactness of the image via the Aubin-Lions theorem \cite[Theorem II.$5.16$]{BoFa12}. \newline
To apply the Schauder fixed point theorem, we will need the following lemma which considers \eqref{eq:heat} where $\mu[u]$ is replaced by a bounded function of time.

\begin{lem}\label{lem:fixed}
		Let $0<\sigma < \frac{2}{(d-2)^+}$, $u_0\in H^1_0(\Omega)$, $\lambda \in L^\infty(\R)$. Then there exists $0<T=T(\|u_0\|_{H^1_0}, \|\lambda\|_{L^\infty})$ such that the Cauchy problem 
		\begin{equation}\label{eq:heat_mu}
			\begin{cases}
				& \partial_t u = \Delta u + g |u|^{2\sigma} u + \lambda u, \\
				& u_{|\partial\Omega} = 0,\\
				& u(0) = u_0,
			\end{cases}
		\end{equation}	
		admits a unique local-in-time solution $u \in C([0,T], H^1_0(\Omega))$. 
        Moreover, if $u_0 \in H^2(\Omega)$, then  $u \in C([0,T], H^1_0(\Omega))$.
\end{lem}

\begin{proof}

		The proof follows from the contraction principle. Let $T,M > 0$, to be chosen later, let $\rho = 2\sigma + 2$ and $\gamma\in \R$ be such that the pair $(\gamma,\rho)$ satisfies condition \eqref{eq:qrcon}, i.e.
		\begin{equation*}
 \frac 2 \gamma + \frac d \rho = \frac d 2.
 \end{equation*}
		We define the space
		\begin{equation*}
		\begin{split}
		 	\mathcal{Y} = \{ u \in L^\infty([0,T), H^1_0(\Omega)) \cap L^\gamma([0,T), W^{1,\rho}(\Omega)), \\ \| u\|_{L^\infty([0,T),H^1)} \leq M, \, \| u\|_{L^\gamma([0,T),W^{1,\rho})} \leq M\},
		\end{split}
		\end{equation*}
		endowed with the following distance
		\begin{equation}\label{eq:dist}
		 d(u,v) = \| u - v\|_{L^\infty([0,T),L^2)} + \| u - v\|_{L^\gamma([0,T),L^\rho)}.
		\end{equation}
		Notice that $(\mathcal Y, d)$ defines a complete metric space. 
		Given $u_0 \in H^1_0(\Omega)$ and $\lambda \in L^\infty(\R)$, we consider the map
\begin{equation}\label{eq:fix_map}
\mathcal G (u) = e^{t\Delta} u_0 + \int_0^t e^{(t-s)\Delta} 
(g |u(s)|^{2\sigma} u(s) + \lambda(s) u(s))ds.
\end{equation}
Our goal is to prove that $\mathcal G:\mathcal Y \rightarrow \mathcal Y$ is a contraction with the distance defined in \eqref{eq:dist}.
By using the space-time estimates \eqref{eq:sptime1}, \eqref{eq:sptime2} and then Holder's inequality and Sobolev's embedding, it follows that there exist two constants $K_0 >0$ and $K(M, \| \lambda\|_{L^\infty})>0$ such that
\begin{equation*}
\begin{split}
 &	\| \mathcal G (u)\|_{L^\infty([0,T),H^1)} + \| \mathcal G (u) \|_{L^\gamma([0,T),W^{1,\rho})} \\
 	&\leq K_0\| u_0\|_{H^1} + \left\| \int_0^t e^{(t - s) \Delta} (g |u(s)|^{2\sigma} u(s) + \lambda(s) u(s)) ds \right\|_{L^\gamma([0,T],W^{1,\rho})} \\ 
 	&\leq K_0\| u_0\|_{H^1} + C\left( \| |u|^{2\sigma} u\|_{L^{\gamma'}([0,T],W^{1,\rho'})} + \|\lambda\|_{L^\infty[0,T]} \| u\|_{L^{1}([0,T],H^1)}\right) \\ 
 	&\leq K_0\| u_0\|_{H^1} + CT^{\frac{\gamma -\gamma'}{\gamma \gamma'}}\| u\|^{2\sigma + 2}_{L^{\infty}([0,T],L^\rho)} \| u\|_{L^{\gamma}([0,T],W^{1,\rho})} \\ 
 	&+ CT \|\lambda\|_{L^\infty[0,T]} \| u\|_{L^{\infty}([0,T],H^1)}\\
 &\leq K_0\| u_0\|_{H^1} + K\left(T + T^{\frac{\gamma -\gamma'}{\gamma \gamma'}}\right)M . 
\end{split}
\end{equation*}	
	Moreover, for any $u,v \in \mathcal Y$, we deduce that	
\begin{equation*}
		d \left( \mathcal G (u), \mathcal G (v)\right) \leq K\left(T + T^{\frac{\gamma - \gamma'}{\gamma \gamma'}}\right) d(u,v).
\end{equation*}	
Let us notice that 
\begin{equation*}
 \frac{\gamma - \gamma'}{\gamma \gamma'} = \frac{2 - (d-2) \sigma}{(2\sigma + 2)} > 0.
\end{equation*}
Thus, given any $u_0 \in H^1_0(\Omega)$, we choose
\begin{equation*}
 M=2K_0\|u_0\|_{H^1}
\end{equation*}
		and we choose $T$ small enough so that
		\begin{equation*}
		 \left(T + T^{\frac{\gamma - \gamma'}{\gamma \gamma'}}\right) K \leq \frac{1}{2}.
		\end{equation*}
 Note that $T$ depends only on $\| u_0\|_{H^1} $ and $\| \mu \|_{L^\infty[0,T)}$. It follows that $\mathcal G (u) \in \mathcal{Y}$ and 
 \begin{equation*}
 d \left( \mathcal G (u), \mathcal G (v)\right) \leq\frac12 d(u,v).
 \end{equation*}
 In particular, $\mathcal G$ is a contraction in $\mathcal{Y}$. By Banach's fixed-point Theorem, $\mathcal G$ has a unique fixed point $u \in \mathcal Y$. $u$ solves the integral equation 
 \begin{equation*}
 u= e^{t\Delta} u_0 + \int_0^t e^{(t-s)\Delta} (g |u(s)|^{2\sigma} u(s) + \lambda(s) u(s))ds,
 \end{equation*}
 and $u \in C([0,T], H^1_0(\Omega))$. \newline
 The persistence of regularity is a standard result.
\end{proof}

For any $\alpha > 0$, we define
\begin{equation*}
    \mu_{\alpha}[u] = \frac{\| \nabla u \|_{L^2}^2 - g\| u\|_{L^{2\sigma + 2}}^{2\sigma + 2}}{\alpha}.
\end{equation*}
The following corollary immediately follows from the previous lemma.
\begin{cor}\label{prp:it_time}
		Let $0 < \sigma < \frac{2}{(d-2)^+}$. Let $\lambda \in L^\infty(\R)$. Let $u_0 \in H^1_0(\Omega)$ and let $u \in C([0,T], H^1_0(\Omega))$ be the corresponding local solution to \eqref{eq:heat_mu}. Then there exists a time $0<T^*=T^*(\|u_0\|_{H^1},\| \lambda\|_{L^\infty}) \leq T$, so that
		\begin{equation}\label{eq:con}
			\sup_{t \in [0,T^*]} \| u(t)\|_{H^1} \leq 2 \| u_0\|_{H^1}. 
		\end{equation} 
		% and
		% \begin{equation}\label{eq:con2}
		% 	\inf_{t \in [0,T^*]} \| u(t) \|_{L^2} \geq \frac{1}{2} \| u_0\|_{L^2}
		% \end{equation}
		Moreover, there exists a constant $\mathcal{C} >0$, depending only on $\| u_0\|_{H^1}$, so that
		\begin{equation}\label{eq:mub}
			\sup_{[0,T^*]} \mu_\alpha[u(t)] \leq \mathcal{C}(\| u_0\|_{H^1}). 
		\end{equation} 
    Finally, if $u_0 \in H^2(\Omega) \cap H^1_0(\Omega) $ then
    \begin{equation}\label{eqCont2}
        \sup_{t \in [0,T^*]} \| u(t)\|_{H^2} \leq 2 \| u_0\|_{H^2}. 
    \end{equation}
\end{cor}
	
\begin{proof}
	The time of existence in Lemma \ref{lem:fixed} depends only on $\| u_0\|_{H^1} $ and $\| \lambda \|_{L^\infty[0,T]}$ and $u \in C([0,T], H^1_0(\Omega))$. So by continuity, we can find a time
	\begin{equation*}
		 T_1(\| u_0\|_{H^1},\| \lambda \|_{L^\infty[0,T_1]}) >0
	\end{equation*}
	such that \eqref{eq:con} is true. 
 % By multiplying equation \eqref{eq:heat_mu} by $u$ and integrating in space and time we obtain
	% 	\begin{equation}\label{eq:tok5}
	% 	 \| u(t) \|_{L^2} - \| u_0\|_{L^2} = 2 \int_0^t g \| u(s)\|_{L^{2\sigma + 2}}^{2\sigma + 2} - \| \nabla u(s) \|_{L^2}^2 + \lambda(s) \| u(s)\|_{L^2}^2 \, ds.
	% 	\end{equation}
	% 	Since $u \in C([0,T], H^1_0(\Omega))$ and the right-hand-side of \eqref{eq:tok5} is a $L^1_{loc}(\R)$ function of time. Hence there exists $T_2(\| u_0\|_{H^1},\| \lambda \|_{L^\infty[0,T_2]}) >0 $ such that condition \eqref{eq:con2} is true.
		The inequality \eqref{eq:mub} is a consequence of Gagliardo-Nirenberg's inequality 
		\begin{equation}\label{eq:mu_in}
			\mu_\alpha[u] \lesssim \| u \|^2_{H^1} + \| u\|^{2 + 2\sigma}_{H^1}.
		\end{equation}
		% together with \eqref{eq:con} and \eqref{eq:con2}.
	\end{proof}

Next, using Schauder fixed point theorem, we prove that a solution to \eqref{eq:heat} exists for $u_0 \in H^2(\Omega) \cap H^1_0(\Omega) $ and $\mu[u]$ is substituted by $\mu_\alpha[u] > 0$, with $\alpha = \| u_0 \|_{L^2}^2$. We will then show that the two Cauchy problems are equivalent. 
% The Schauder fixed point theorem can be found in \cite{GiTr98}, see Corollary $11.2$. Aubin-Lions lemma may be found in \cite[Theorem II.$5.16$]{BoFa12}.
\begin{prop}\label{prpSchauder}
    For any $0 \neq u_0 \in H^2(\Omega) \cap H^1_0(\Omega)$, there exists a solution $u\in C([0,T),H^2(\Omega) \cap H^1_0(\Omega))$ to 
    \begin{equation}\label{eqMuAlpha}
    \begin{cases}
        \partial_t u = \Delta u + g|u|^{2\sigma} u + \mu_{\alpha}[u] u, \\
        u(0) = u_0,
    \end{cases}
\end{equation}
where $\| u_0 \|_{L^2}^2 = \alpha$.
\end{prop}
\begin{proof}
Given any $u_0 \in H^2(\Omega) \cap H^1_0(\Omega) $, $u_0 \neq 0$, let $F_{u_0}: C([0,T_{u_0}], H^1_0(\Omega)) \to C([0,T_{u_0}], H^1_0(\Omega))$ be the map $F_{u_0}(u) = v$ where $v$ is a solution to the Cauchy problem 
\begin{equation*}
    \begin{cases}
        \partial_t v = \Delta v + g|v|^{2\sigma} v + \mu_{\alpha}[u] v, \\
        v(0) = u_0,
    \end{cases}
\end{equation*}
where $\| u_0 \|_{L^2}^2 = \alpha$, see Lemma \ref{lem:fixed}.
  Let 
\begin{equation*}
    B_{u_0} = \left\{ u \in C([0,T_{u_0}], H^1_0(\Omega))\, : \, \| u \|_{L^\infty([0,T_{u_0}],H^1)} \leq 2 \| u_0 \|_{H^1} \right\}
\end{equation*}
Let $F_{u_0}(B_{u_0}) =  I_{u_0}$.  Notice that, for any $u \in B_{u_0}$, $F_{u_0}(u) \in C([0,T_{u_0}], H^2(\Omega) \cap H^1_0(\Omega) )$. By \eqref{eqCont2}, there exists a time $\tilde T = T_{u_0} >0 $ such that $I_{u_0} \subset B_{u_0}$ and for any $u\in B_{u_0}$, 
\begin{equation*}
    \| F_{u_0}(u) \|_{L^\infty([0,\tilde T], H^2)} \leq 2 \| u_0\|_{H^2}.
\end{equation*} 

In order to apply the Schauder fixed point theorem, we shall prove that $I_{u_0}$ is precompact. We observe that for any  $u \in B_{u_0}$, $ v = F_{u_0}(u) \in W$ where
    \begin{equation*}
        W = \left\{ v \in C([0,\tilde{T}], H^2(\Omega) \cap H^1_0(\Omega) )\,:\, \partial_t v \in L^2([0,\tilde{T}], L^2(\Omega)) \right\}.
    \end{equation*} 
 %    From Aubin-Lions lemma, the embedding of $W$ in $C([0,T], H^1_0(\Omega))$ is compact. Thus, there exists a subsequence $\{v_{n_k}\}_{k \in \N}$ and $v_\infty \in L^{\infty}([0,\tilde{T}], H^2(\Omega) \cap H^1_0(\Omega) )$ such that
	% \begin{equation*}
	% 	v_{n_k} \overset{\ast}{\rightharpoonup} v_\infty  \ \mbox{ in } \ L^{\infty}([0,\tilde{T}], H^1_0(\Omega)) \quad \mbox{ and } \quad 	v_{n_k}  \rightarrow v_\infty  \ \mbox{ in } \ C([0,\tilde{T}], H^1_0(\Omega)). 
	% \end{equation*}
    From the Aubin-Lions lemma, the embedding of $W$ in $C([0,T], H^1_0(\Omega))$ is compact.
    Since $I_{u_0} \subset W$, we obtain that $I_{u_0} $ is indeed precompact. Applying the Schauder fixed point theorem, we find that $F_{u_0}$ admits a fixed point $u \in C([0,T], H^1_0(\Omega))$, a solution to 
    \begin{equation}\label{eqHeatMuA}
        \begin{cases}
        \partial_t u = \Delta u + g|u|^{2\sigma} u + \mu_{\alpha}[u]u, \\
        u(0) = u_0 \in H^2(\Omega) \cap H^1_0(\Omega) .
        \end{cases}
    \end{equation}
    By the standard persistence of regularity, we conclude that $u\in C([0,T], H^2(\Omega) \cap H^1_0(\Omega) )$.
    \end{proof}
    Next, we employ a density argument to prove that we can find solutions to \eqref{eqHeatMuA} with initial data in $H^1_0(\Omega) \setminus H^2(\Omega) \cap H^1_0(\Omega) $.
    \begin{theorem}\label{thmLocExisDom}
        Let $u_0 \in H^1_0(\Omega)$, $\alpha = \| u_0 \|_{L^2}^2 >0$. Then there exists a solution $u \in C([0,T],H^1_0(\Omega))$ to \eqref{eqMuAlpha}.
    \end{theorem}
    
    \begin{proof}
    Let $u_0 \in H^1_0(\Omega) $ and let $\{u_0^{(n)}\}_{n \in \N} \subset H^2(\Omega) \cap H^1_0(\Omega) $ be such that 
    \begin{equation*}
      \lim_{n \to \infty}  \| u_0^{(n)} - u_0 \|_{H^1} = 0.
    \end{equation*}
    We can suppose that there $C >0$ such that
    \begin{equation}\label{eqSupInCon}
        \sup_{n \in \N} \| u_0^{(n)} \|_{H^1} \leq C, \quad \| u_0^{(n)} \|_{L^2} = \alpha.
    \end{equation}
    We denote by $\{u^{(n)}\}_{n \in N} \subset C([0,T_n],H^2(\Omega) \cap H^1_0(\Omega) )$ the solutions to \eqref{eqHeatMuA} emanating from $u_0^{(n)}$. Using Corollary \ref{prp:it_time} and \eqref{eqSupInCon}, we see that there exists $T >0$, such that
    $$\{u^{(n)}\}_{n \in N} \subset C([0,T],H^2(\Omega) \cap H^1_0(\Omega) )$$ 
    \begin{equation*}
        \sup_{n\in \N, t\in [0,T]} \| u^{(n)} \|_{H^1} \leq 2 C,
    \end{equation*}
    and 
    \begin{equation*}
        \sup_{n\in \N} \| \mu_\alpha[u^{(n)}] \|_{L^\infty[0,T]} \leq C_1,
    \end{equation*}
    where $C_1(C) > 0$. We observe that $\{u^{(n)}\}$ is uniformly bounded also in $L^2([0, T],H^2(\Omega))$. Indeed, we multiply equation \eqref{eqHeatMuA} by $\Delta u^{(n)}$ and integrate in space to get
	\begin{equation*}
	 \begin{aligned}
	 -\frac{1}{2} \frac{d}{dt} \| \nabla u^{(n)}\|_{L^2}^2 &= \| \Delta u^{(n)}\|_{L^2}^2 \\& - (2\sigma + 1)g \int_\Omega |u^{(n)}|^{2\sigma} |\nabla u^{(n)}|^2 dx - \mu_\alpha[u^{(n -1)}] \| \nabla u^{(n)}\|_{L^2}^2 dx.
	 \end{aligned}
	\end{equation*}
	Integrating this equation in time leads to
	\begin{equation*}
 \begin{aligned}
 \int_0^{T} \| \Delta u^{(n)}\|_{L^2}^2 \, ds &\lesssim 1 + \|\mu_\alpha [u^{(n -1)}] \|_{L^\infty([0,T])}T \| \nabla u^{(n)}\|_{L^{\infty}([0,T],L^2)}\\
 &+ \int_0^{T} \int_\Omega |u^{(n)}|^{2\sigma} |\nabla u^{(n)}|^2 \,dx \,ds.
 \end{aligned}
	\end{equation*}
	We suppose that $d \geq 3$. The case $d \leq 2$ is similar. By Holder's inequality, we see that
	\begin{equation*}
		\int_\Omega |u^{(n)}|^{2\sigma} |\nabla u^{(n)}|^2 dx \leq \| u^{(n)}\|_{L^\frac{2d}{d-2}}^{2\sigma} \| \nabla u^{(n)} \|_{L^r}^2,
	\end{equation*}
	where 
	\begin{equation*}
		r = \frac{2d}{d-\sigma(d-2)}.
	\end{equation*}
 Using H\"older's inequality in time yields
 \begin{equation*}
 \int_0^{T} \int_\Omega |u^{(n)}|^{2\sigma} |\nabla u^{(n)}|^2 \, dx \, ds \leq \|u^{(n)} \|_{L^\infty([0,T],H^1)}^{2\sigma} (T)^{1/p} \| \nabla u^{(n)}\|_{L^q([0,T],L^r)}^2,
	\end{equation*}
 where 
 \begin{equation*}
 q = \frac{4}{\sigma(d-2)}
 \end{equation*}
 is such that $(q,r)$ satisfies \eqref{eq:qrcon} and 
 \begin{equation*}
 0 < p = \frac{q}{q-2} < \infty.
 \end{equation*}
	Consequently, we obtain 
    $$\{u^{(n)}\}_{n\in \N} \subset L^{\infty}([0,T], H^1_0(\Omega)) \cap L^2([0,T], H^2(\Omega) \cap H^1_0(\Omega)).$$ From the embedding $H^2 \hookrightarrow L^{4\sigma + 2}$, we also have $\{ \partial_t u^{(n)}\}_{n\in \N}  \subset L^2([0,T], L^2(\Omega))$ and the two sequences are uniformly bounded in these spaces.
\newline	
So there exists a subsequence $\{u^{(n_k)}\}_{k \in \N}$, and $u \in L^{\infty}([0,T], H^1_0(\Omega)) \cap L^{2}([0,T], H^2(\Omega) \cap H^1(\Omega))$ with $\partial_t u \in L^{2}([0,T], L^2(\Omega))$, such that
	\begin{equation*}
		u^{(n_k)} \overset{\ast}{\rightharpoonup} u \ \mbox{ in } \ L^{\infty}([0,T], H^1_0(\Omega)) \quad \mbox{ and } \quad 	u^{(n_k)} \rightharpoonup u \ \mbox{ in } \ L^{2}([0,T], H^2(\Omega)). 
	\end{equation*}
	The Aubin-Lions lemma implies that $u^{(n_k)} \rightarrow u$ strongly in $L^2([0,T], H^1_0(\Omega))$ and in $C([0,T], L^2(\Omega))$. Therefore, $\mu_\alpha[u^{(n_k)}] \rightarrow \mu_\alpha[u]$ strongly in $L^2([0,T])$. Thus, $u$ satisfies weakly (and strongly) \eqref{eqHeatMuA}. 
 \end{proof}
    As a last step, we prove the equivalence of the Cauchy problem  \eqref{eqMuAlpha} with \eqref{eq:heat}. 
    % \begin{equation}\label{eqHeatMuAH1}
    %     \begin{cases}
    %         \partial_t u = \Delta u + g|u|^{2\sigma} u + \mu_{\alpha}[u]u, \\
    %         u(0) = u_0 \in H^1_0(\Omega)
    %     \end{cases}
    % \end{equation}
    % with \eqref{eq:heat}. 
    \begin{prop}\label{prpEquivalence}
        The Cauchy problem \eqref{eqMuAlpha} is equivalent to  \eqref{eq:heat}.
    \end{prop}
    \begin{proof}
    Taking formally the scalar product of \eqref{eqMuAlpha} with $u$ and integrating in time, we see that solutions satisfy
    \begin{equation*}
        \| u(t) \|_{L^2}^2 = (\| u(0) \|_{L^2}^2 - \alpha) e^{\frac{2}{\alpha} \int_0^t (\|\nabla u(s)\|_{L^2}^2 - g\| u(s) \|_{L^{2\sigma + 2}}^{2\sigma + 2}) ds} + \alpha.
    \end{equation*}
    Since $\alpha = \| u(0)\|_{L^2}^2$, we obtain that the $L^2$-norm is constant in time.
     \end{proof}

Since Theorem \ref{thmLocExisDom} does not imply uniqueness of solutions,  we will show it in the following Proposition.

\begin{prop}\label{prp:uni}
			 For any $u_0\in H^1_0(\Omega)$ there exists a unique strong solution to \eqref{eq:heat}.
\end{prop}

\begin{proof}
Let $u_0 \in H^1_0(\Omega)$ so that $\|u_0\|_{H^1} =M$ and suppose that $$u,v \in C([0,T],H^1_0(\Omega))$$ are two different strong solutions to \eqref{eq:heat}. Here we choose $T > 0$ so that if $T_u, T_v > 0$ are the maximal times of existence of $u,v$ then $T < \min (T_u,T_v)$. In particular, there exists a constant $C = C(M) >0$ such that 
			\begin{equation}
				\sup_{t\in[0,T]} \|u\|_{H^1} \leq C, \quad \sup_{t\in[0,T]}\|v\|_{H^1} \leq C,
			\end{equation}
			and
			\begin{equation}
				\sup_{t\in[0,T]}\mu[u] \leq C, \quad \sup_{t\in[0,T]} \mu[v] \leq C.
			\end{equation}
By multiplying the equation 
\begin{equation*}
 \partial_t (u -v) = \Delta(u-v) + g |u|^{2\sigma} u - g |v|^{2\sigma} v+ \mu[u] u - \mu[v]v
\end{equation*}
by $(u-v)$ and integrating in space, we obtain
\begin{equation}\label{eq:took1}
\begin{split}
 \frac{1}{2} \frac{d}{dt} \| u -v \|_{L^2}^2 &= -\| \nabla (u - v)\|_{L^2}^2 + \int_\Omega (\mu[u] u - \mu[v]v)(u-v) dx \\
 &+ g \int_\Omega (|u|^{2\sigma} u - |v|^{2\sigma}v)(u-v) dx.
 \end{split}
\end{equation}
We start by estimating the contribution given by the power-type nonlinearities. We suppose that $d \geq 2$. Notice that
\begin{equation*}
 \begin{split}
 \left| \int_\Omega (|u|^{2\sigma} u - |v|^{2\sigma }v)(u - v) \, dx \right| & \lesssim \int_\Omega (|u|^{2\sigma} + |v|^{2\sigma}) | u - v|^2 \, dx \\
 & \lesssim \left(\| u\|_{L^{2\sigma p '}}^{2\sigma} + \| v\|_{L^{2\sigma p '}}^{2\sigma} \right) \| u - v\|_{L^{2p}}^2.
 \end{split}
 \end{equation*}
 We choose the exponent $p$ so that
 \begin{equation*}
 p' = \frac{d}{\sigma(d-2)}, \quad p = \frac{d}{2\sigma + d(1 - \sigma)},
 \end{equation*}
and consequently
 \begin{equation*}
 2\sigma p' = \frac{2d}{d-2}.
 \end{equation*}
By Gagliardo-Nirenberg and Young's inequalities, for any $\delta >0$, we obtain
\begin{equation*}
 \| u - v\|_{L^{2p}}^2 \lesssim \| u - v\|_{L^2}^{2 - \sigma(d-2)} \|\nabla( u - v)\|_{L^2}^{\sigma(d-2)} \lesssim \frac{1}{\delta} \| u - v\|_{L^2}^2 + \delta \|\nabla( u - v)\|_{L^2}^2. 
\end{equation*}
Thus, it follows that for any $\delta > 0$, we have
\begin{equation}\label{eq:tok10}
 \left| \int_\Omega (|u|^{2\sigma} u - |v|^{2\sigma }v)(u - v) dx \right| \lesssim \frac{1}{\delta} \| u - v\|_{L^2}^2 + \delta \|\nabla( u - v)\|_{L^2}^2.
\end{equation}
Next, we deal with the nonlocal term. We observe that 
\begin{equation*}
\begin{split}
 &\int_\Omega (\mu[u] u - \mu[v]v)(u-v) dx = (\mu[u] - \mu[v])\int_\Omega u (u-v)dx + \mu[v] \|u - v\|_{L^2}^2 \\
 & \leq \left( \left| \|\nabla u \|_{L^2}^2 - \| \nabla v\|_{L^2}^2 \right| + \left| \| u \|_{L^{2\sigma+ 2}}^{2\sigma +2} - \| v \|_{L^{2\sigma+ 2}}^{2\sigma +2} \right| \right) \| u - v\|_{L^2} + C(M) \|u - v\|_{L^2}^2.
\end{split}
\end{equation*}
By Young's inequality, for any $\delta > 0$ we have
\begin{equation*}
 \left| \|\nabla u \|_{L^2}^2 - \| \nabla v\|_{L^2}^2 \right| \| u - v\|_{L^2} \leq C(M) \left( \frac{1}{\delta} \| u - v\|_{L^2}^2 + \delta \|\nabla( u - v)\|_{L^2}^2 \right).
\end{equation*}
From \eqref{eq:tok10} and the triangular inequality, we also obtain that, for any $\delta >0$, 
\begin{equation}\label{eq:took2}
 \left|\| u \|_{L^{4\sigma + 2}}^{2\sigma + 1} - \| v \|_{L^{4\sigma + 2}}^{2\sigma + 1}\right| \| u-v\|_{L^2} \lesssim \frac{1}{\delta} \| u - v\|_{L^2}^2 + \delta \|\nabla( u - v)\|_{L^2}^2. 
\end{equation}
In particular, from \eqref{eq:took1}, \eqref{eq:took2} and \eqref{eq:tok10} we have
\begin{equation*}
\begin{split}
 \frac{d}{dt} \| u - v\|_{L^2}^2 \leq (C(M)\delta - 1) \| \nabla (u - v) \|_{L^2}^2 + \frac{C(M)}{\delta} \| u - v\|_{L^2}^2.
\end{split}
\end{equation*}
For $\delta < \frac{1}{2C(M)}$, we obtain by Gronwall's lemma that
\begin{equation*}
 \| u(t) - v(t)\|_{L^2}^2 \leq \| u(0) - v(0)\|_{L^2}^2 e^{C t}.
\end{equation*}
Since $u(0) = v(0)$ it follows that $u =v$ for all $t \in [0,T]$.
	\end{proof} 
\begin{proof}[Proof of Theorem \ref{thm:proof2}:]
Theorem \ref{thmLocExisDom} implies that for any initial condition $u_0 \in H^1_0(\Omega)$, there exists a local strong solution to \eqref{eq:heat}. Moreover, it is possible to extend the time of existence until the $H^1$-norm of the solution is bounded, giving the blow-up alternative \eqref{eq:bu1}. On the other hand, Proposition \ref{prp:uni} implies that for any initial condition $u_0 \in H^1_0(\Omega)$, there exists at most one solution.
\end{proof}

\begin{remark}
Let us note that the approach developed in Theorem \ref{thmLocExisDom} cannot be used in the case $\Omega=\R^d$. Indeed, at several points, the use of the Aubin-Lions lemma required the compact embedding $H^1(\Omega) \hookrightarrow L^2(\Omega)$.
\end{remark}

\subsection{The case $0<\sigma < \frac{1}{(d-2)^+}$}

\begin{theorem}\label{thm:proof1}
		Let $0< \sigma < \frac{1}{(d-2)^+}$. Let either $\Omega$ be a regular bounded domain or $\Omega = \R^d$. For any $u_0 \in H^1_0(\Omega)$, there exists a maximal time of existence $T_{max}>0$ and a unique solution $u \in C([0,T_{max}), H^1_0(\Omega))$ to equation \eqref{eq:heat}. Moreover, either $T_{max} = \infty$, or $T_{max} < \infty$ and we have
		\begin{equation}\label{eq:bu}
			\lim_{t \rightarrow T_{max}} \| \nabla u(t)\|_{L^2} = \infty.
		\end{equation}
\end{theorem}

We need the following lemma, where the restriction on $\sigma$ is justified by direct application of Sobolev's embedding in \eqref{eq:lem2} below.

\begin{lem}
			Let $u,v \in H^1_0(\Omega)$ with $ \| u\|_{H^1}, \| v\|_{H^1} \leq M.$ Let $\sigma < \frac{1}{(d - 2)^+}$
			Then there exists a constant $C = C(M) >0$ such that
			\begin{equation}\label{eq:lem1}
				\left| \|\nabla u \|_{L^2}^2 - \| \nabla v\|_{L^2}^2 \right| + \left| \| u \|_{L^{2\sigma+ 2}}^{2\sigma +2} - \| v \|_{L^{2\sigma+ 2}}^{2\sigma +2} \right| \leq C \| u-v\|_{H^1},
			\end{equation}
	\begin{equation}\label{eq:lem2}
	\left\| |u|^{2\sigma} u - |v|^{2\sigma} v \right\|_{L^2} \leq C \| u-v\|_{H^1}.
\end{equation}
\end{lem}

\begin{proof}
	The first term on the left-hand side of \eqref{eq:lem1} is bounded by
	 \begin{equation*}
	 \int_\Omega |\nabla u|^2 - |\nabla v|^2 dx \leq \left(\|\nabla u \|_{L^2} + \|\nabla v \|_{L^2}\right) \| \nabla(u-v)\|_{L^2}.
	 \end{equation*}
	 To obtain the bound on the second term on the left-hand-side of \eqref{eq:lem1}, we observe that
	 \begin{equation*}
	 \begin{split}
	 \int_\Omega \left||u|^{2\sigma +2} - |v|^{2\sigma +2} \right| dx &\lesssim \int_\Omega |u - v| (|u|^{2\sigma + 1} + |v|^{2\sigma + 1}) dx \\ & \lesssim \left(\| u \|_{L^{4\sigma + 2}}^{2\sigma + 1} + \| v \|_{L^{4\sigma + 2}}^{2\sigma + 1}\right) \| u-v\|_{L^2}. 
	 \end{split}
	 \end{equation*}
	 The condition $0<\sigma < \frac{1}{(d-2)^+}$ implies that $2 < 4\sigma + 2 < \frac{2d}{(d-2)^+}$. Thus Sobolev's embedding theorem implies that there exists a constant $C(M)>0$ such that 
	 \begin{equation*}
	 \int_\Omega \left||u|^{2\sigma +2} - |v|^{2\sigma +2} \right| dx \leq C(M) \| u-v\|_{L^2}. 
	 \end{equation*}
	 Next, we obtain the bound in \eqref{eq:lem2}. Suppose that $d >2$. Then, by H\"older's inequality, we get
	 \begin{equation*}
	 \int_\Omega \left||u|^{2\sigma} u - |v|^{2\sigma} v\right|^2 \, dx\lesssim \left(\int_\Omega |u - v|^{2p'} \,dx \right)^\frac{1}{p'} \left(\int_\Omega |u + v|^{4\sigma p} \,dx \right)^\frac{1}{p},
	 \end{equation*}
	 Let us choose
	 \begin{equation*}
	 p = \frac{d}{2\sigma(d-2)}, \quad p' = \frac{d}{d - 2\sigma(d-2)}.
	 \end{equation*}
	 From $\sigma < \frac{1}{d-2}$ we have 
	 \begin{equation*}
	 4\sigma p = \frac{2d}{d-2}, \quad 2p' \leq \frac{2d}{d-2}.
	 \end{equation*}
	 Inequality \eqref{eq:lem2} follows from Sobolev's embedding theorem. Finally, observe that the case $d \leq 2$ follows from a similar and easier proof. 
\end{proof}

\begin{proof}[Proof of Theorem \ref{thm:proof1}:]
	Let $T, M, N >0$ be some constants that will be chosen later and consider the set
	\begin{equation*}
		\mathcal{X} = \{ u \in L^\infty([0,T], H^1_0(\Omega)),\, \| u\|_{L^\infty([0,T],H^1)} \leq 2M, \, \inf_{[0,T]} \| u\|_{L^2} \geq \frac{N}{2} \},
	\end{equation*}
	equipped with the natural distance 
	\begin{equation*}
		d(u,v) = \| u - v\|_{L^\infty([0,T],H^1)}. 
	\end{equation*}
	Clearly, $(\mathcal{X}, d)$ is a complete metric space. Given any $u_0 \in H^1_0(\Omega)$, $u_0 \centernot{\equiv} 0$, $u,v\in \mathcal X$, we also define the map $\mathcal{F}(u_0)(u)(t) = \mathcal{F}(u)(t)$ as
	\begin{equation}\label{eq:int}
		\mathcal{F}(u)(t) = e^{t\Delta} u_0 + \int_0^t e^{(t-s)\Delta} (g |u|^{2\sigma} u + \mu[u] u)ds.
	\end{equation}
 Our goal is to prove that $\mathcal F$ is a contraction in the space $\mathcal{X}$. 
 First of all, we observe that, by Gagliardo-Nirenberg's inequality, for any $u\in\mathcal X$ there exists a constant $C(M, N) > 0$ such that
		\begin{equation}\label{eq:omegaL}
			\| \mu[u] \|_{L^\infty[0,T]} \lesssim \frac{1}{N^2} \left(\| u\|_{L^\infty([0,T],H^1)}^2 + \| u\|_{L^\infty([0,T],H^1)}^{2\sigma + 2} \right) \leq C.
		\end{equation} 
 Let us show that $\mathcal F$ maps the set $\mathcal X$ into itself. We observe that there exists $C_1>0$ such that
			\begin{equation}\label{eq:strHeat11}
				\begin{split}
					\|	\mathcal{F}(u)\|_{L^\infty([0,T],L^2)} &\leq \|u_0\|_{L^2} + T\| g|u|^{2\sigma} u + \mu[u] u)\|_{L^\infty([0,T],L^2)} \\ 
 &\leq \|u_0\|_{L^2} + C_1 T \big( \| u \|_{L^\infty([0,T],L^{4\sigma + 2})}^{2\sigma + 1} + \| \mu[u]\|_{L^\infty[0,T]} \| u\|_{L^\infty([0,T],L^2)}\big).
				\end{split}
			\end{equation}
			Notice that, for $0<\sigma < \frac{1}{(d-2)^+}$, we have $2<4\sigma + 2 < \frac{2d}{(d-2)^+}$. Thus, by Sobolev's embedding theorem, it follows that there exists a constant $C_2(M,N) > 0$ such that
			\begin{equation}\label{eq:strHeat12}
				\|	\mathcal{F}(u)\|_{L^\infty([0,T],L^2)} \leq \|u_0\|_{L^2} + C_2 T.
			\end{equation}
 Moreover, by using \eqref{eq:linmass}, \eqref{eq:strHeat11} and \eqref{eq:strHeat12} we also see that there exists $C_3 >0$ such that
 \begin{equation*}
 \begin{aligned}
 \| \mathcal F(u)\|_{L^2} &\geq \| e^{t\Delta} u_0\|_{L^2} - \left\| \int_0^t e^{(t-s)\Delta} (g |u|^{2\sigma} u + \mu[u] u) \, ds \right\|_{L^2} \\
 &\geq \bigg( \| u_0\|_{L^2}^2 - 2\int_0^t \| \nabla e^{s \Delta} u_0\|_{L^2}^2 \, ds \bigg)^\frac{1}{2} \\ 
 & - C_1 T \big( \| u \|_{L^\infty([0,T],L^{4\sigma + 2})}^{2\sigma + 1} + \| \mu[u]\|_{L^\infty[0,T]} \| u\|_{L^\infty([0,T],L^2)}\big) \\
 & \geq \bigg( \| u_0\|_{L^2}^2 - C_3 T \| \nabla u_0\|_{L^2}^2 \bigg)^\frac{1}{2} - C_2T.
 \end{aligned}
 \end{equation*}
 In particular, it follows that there exists $K(M,N) >0$ such that
 \begin{equation}\label{eq:strHeat13}
 \inf_{t \in [0,T]} \| \mathcal F(u)\|_{L^2} \geq \| u_0\|_{L^2} - K (T^\frac{1}{2} + T).
 \end{equation}
	Furthermore, by the smoothing property of the heat semigroup for the gradient stated in \eqref{eq:smooth}, we infer that
			\begin{equation*}
				\|	\nabla \mathcal{F}(u)\|_{L^\infty([0,T],L^2)} \lesssim \| u_0\|_{H^1} + T^\frac{1}{2} \| |u|^{2\sigma} u + \mu[u] u \|_{L^\infty([0,T], L^2)}.
			\end{equation*}
	Notice that, since $\sigma < \frac{1}{(d-2)^+}$, there exists a constant $K_1(M,N) >0$ such that
	\begin{equation}\label{eq:strHeat14}
			\|\nabla\mathcal{F}(u)\|_{L^\infty([0,T],L^2)} \leq \|u_0\|_{H^1} + K_1 T^\frac{1}{2}.
	\end{equation}
	We choose $M = \| u_0\|_{H^1}$, $N = \| u_0\|_{L^2}$ and $T >0$ so that 
 \begin{equation}\label{eq:timeHeat1}
 T \leq \frac{M}{C_2}
 \end{equation}
 where $C_2$ is defined in \eqref{eq:strHeat12},
 \begin{equation}\label{eq:timeHeat2}
 T^\frac{1}{2} + T \leq \frac{N}{2K}
 \end{equation}
 where $K$ is defined in \eqref{eq:strHeat13} and 
 \begin{equation}\label{eq:timeHeat3}
 T^\frac{1}{2} \leq \frac{M}{K_1}
 \end{equation}
 where $K_1$ is defined in \eqref{eq:strHeat14}. This implies that $\mathcal{F}:\mathcal{X} \rightarrow \mathcal{X}$.
\newline
Next we show that $\mathcal F$ is a contraction map on $\mathcal X$, namely there exists $0<K_2<1$ such that, for any $u, v\in\mathcal X$, we have
			\begin{equation*}
				\| \mathcal{F}(u) - \mathcal{F}(v)\|_{L^\infty([0,T],H^1)} \leq K_2\| u - v\|_{L^\infty([0,T],H^1)}.
			\end{equation*} 
				From the smoothing effect of the heat semigroup \eqref{eq:smooth} we get
			\begin{equation}\label{eq:tok2}\begin{aligned}
			\| \mathcal{F}(u) - \mathcal{F}(v)\|_{L^{\infty}([0,T], H^1)} &\leq C \left\| |u|^{2\sigma} u - |v|^{2\sigma} v + \mu[u] u - \mu[v] v\right\|_{L^2([0, T],L^2)}\\
			&\leq C T^{\frac12} \left\| |u|^{2\sigma} u - |v|^{2\sigma} v + \mu[u] u - \mu[v] v\right\|_{L^\infty([0, T],L^2)}
		\end{aligned}
		\end{equation}
		As a consequence of inequality \eqref{eq:lem1} we have that, for any $u,v \in \mathcal{X}$ there exists a constant $C_3(M,N) >0$ such that
			\begin{equation}\label{eq:om2}
				\| \mu(u) - \mu(v)\|_{L^{\infty}(0,T)} \leq C_3 \| u-v\|_{L^{\infty}([0,T],H^1)}.
			\end{equation}
	From \eqref{eq:om2} we obtain that
	\begin{equation}\label{eq:tok3}
		\left\| \mu[u] u - \mu[v] v\right\|_{L^{\infty}([0,T], L^2)} \leq C_3 \| u-v\|_{L^{\infty}([0,T],H^1)}. 
	\end{equation}
Moreover, from the inequality \eqref{eq:lem2} we obtain
\begin{equation}\label{eq:tok4}
	\left\| |u|^{2\sigma} u - |v|^{2\sigma} v \right\|_{L^{\infty}([0,T], L^2)} \leq C_3 \| u-v\|_{L^{\infty}([0,T],H^1)}.
\end{equation}
From inequalities \eqref{eq:tok2},\eqref{eq:tok3},\eqref{eq:tok4} we have
\begin{equation*}
	\| \mathcal{F}(u) - \mathcal{F}(v)\|_{L^{\infty}([0,T], H^1)} \leq T^{\frac12} C_3 \| u - v\|_{L^{\infty}([0,T], H^1)}. 
\end{equation*}
By choosing $T$ satisfying conditions \eqref{eq:timeHeat1}, \eqref{eq:timeHeat2}, \eqref{eq:timeHeat3} and such that $T^{\frac12}C_3<1$, we conclude that $\mathcal F$ is a contraction in $\mathcal X$. Thus $\mathcal F$ admits a fixed point $u \in C([0,T],H^1_0(\Omega))$, which is a solution to \eqref{eq:heat}. The uniqueness of this solution follows from the fact that the map $\mathcal F$ is a contraction in $\mathcal X$. Moreover, notice that we can extend the local solution until the $H^1$-norm of $u(t)$ is bounded. Hence we obtain the blow-up alternative, that is, if $T_{max}>0$ is the maximal time of existence, then either $T_{max}= \infty$ or $T_{max}< \infty$ and
			\begin{equation*}
				\lim_{t \rightarrow T_{max}} \| \nabla u(t)\|_{L^2} = \infty.
			\end{equation*}
\end{proof}

\subsection{Local well-posedness for $\sigma \geq \frac{1}{2}$}

In this subsection, we present our third proof for the local well-posedness result, based on a fixed point argument. 
We state the following theorem for the case $\Omega=\R^d$, however, it is straightforward to see that the same proof applies also to the case of bounded domains. The necessity of the restriction $\sigma \geq \frac{1}{2}$ is further commented in the remark \ref{rem:lip} below the proof of the following theorem.

\begin{theorem}\label{thm:lwp}
	Let $\frac{1}{2} \leq \sigma < \frac{2}{(d-2)^+}$. Then, for any $u_0\in H^1(\R^d)$, there exists a maximal time $T_{max}>0$ and a local solution $u \in C([0,T_{max}),H^1(\R^d))$ to \eqref{eq:heat}. Moreover, either $T_{max} = \infty$, or $T_{max} < \infty$ and
\begin{equation*}
\lim_{t \rightarrow T_{max}} \| \nabla u(t)\|_{L^2} = \infty.
\end{equation*}
\end{theorem}

\begin{proof}
Fix $M, N, T>0$, to be chosen later, let $ r=2\sigma +2$ and 
\begin{equation*}
 q = \frac{4 \sigma + 4}{d \sigma}
\end{equation*}
so that the pair $(q, r)$ satisfies condition \eqref{eq:qrcon}. Consider the set
\begin{equation*}
	\begin{aligned}
		 E = \bigg\{& u \in L^{\infty}\left([0,T], H^1(\R^d) \right) \cap L^{q}\left((0,T), W^{1,r} (\R^d) \right) : \, \\
		 & \| u\|_{L^{\infty}\left((0,T), H^1\right)} \leq M, \ \| u\|_{L^{q}\left((0,T), W^{1,r} \right)} \leq M, \ \inf_{t \in [0,T]} \| u \|_{L^2} \geq \frac{N}{2} \bigg\},		 
	\end{aligned}
\end{equation*}
equipped with the distance
\begin{equation*}
	\mathrm{d}(u, v)=\|u-v\|_{L^{q}\left((0, T), W^{1,r}\right)}+\|u-v\|_{L^{\infty}\left([0, T], H^1\right)}.
\end{equation*}
Clearly $(E, d)$ is a complete metric space. Consider now $u,v \in E$ and observe that the inequality 
\begin{equation*}
	\left||u|^{2\sigma}u - |v|^{2\sigma} v \right| \lesssim \left(|u|^{2\sigma}+|v|^{2\sigma}\right)|u-v|,
\end{equation*} 
implies that
\begin{equation*}
	\begin{aligned}
		&\left\||u|^{2\sigma} u - |v|^{2\sigma} v \right\|_{L^{q^{\prime}}\left((0, T), L^{r^\prime}\right)} \lesssim \left(\|u\|_{L^{\infty}\left([0, T], L^{r}\right)}^{2\sigma}+\|v\|_{L^{\infty}\left([0,T], L^{r}\right)}^{2\sigma}\right)\|u-v\|_{L^{q^\prime}\left((0, T), L^{r}\right)}.
	\end{aligned}
\end{equation*}
Moreover, using the embedding $H^{1}(\R^{d}) \hookrightarrow L^{r}( \R^d )$ and H\"older's inequality, we have
\begin{equation*}
	\left\| \nabla \left( |u|^{2\sigma} u - |v|^{2\sigma} v \right) \right\|_{L^{r^{\prime}}} \lesssim \left( \|u\|_{L^{r}}^{2\sigma-1} + \|v\|_{L^{r}}^{2\sigma-1}\right) \left( \|\nabla u\|_{L^{r}} + \|\nabla v\|_{L^{r}} \right) \|u - v\|_{H^1}
\end{equation*}
which implies
\begin{equation}\label{eq:omega2}
	\begin{aligned}
		&\left\|\nabla \left(|u|^{2\sigma} u - |v|^{2\sigma} v \right) \right\|_{L^{q^{\prime}}\left((0, T), L^{r^{\prime}}\right)} \lesssim \left(\|u\|_{L^{\infty}\left([0,T], L^{r}\right)}^{2\sigma-1}+\|v\|_{L^{\infty}\left([0,T], L^{r}\right)}^{2\sigma-1}\right) \\ 
		& \cdot\left( \| \nabla u\|_{L^{q^\prime}((0, T), L^{r})} + \| \nabla v\|_{L^{q^\prime}((0, T), L^{r})} \right)\|u-v\|_{L^{\infty}\left([0,T], H^1\right)}.
	\end{aligned} 
\end{equation}
Using H\"older's inequality in time, we deduce from the above estimates that
\begin{equation*}
	 \left\||u|^{2\sigma}u \right\|_{L^{q^{\prime}}\left((0, T), W^{1, r^{\prime}}\right)} \lesssim \left(T+T^{\frac{q-q^{\prime}}{q q^{\prime}}}\right)\left(1+M^{2\sigma}\right) M
\end{equation*}
 and
\begin{equation*}
	\left\||u|^{2\sigma}u - |v|^{2\sigma} v \right\|_{L^{q^{\prime}}\left((0, T), W^{1, r^{\prime}}\right)}\lesssim \left(T+T^{\frac{q-q^{\prime}}{q q^{\prime}}}\right)\left(1+M^{2\sigma}\right) \mathrm{d}(u, v) .
\end{equation*}
Next, we observe that, for any $u,v \in E$, from 
\begin{equation*}
	\left| \mu[u] u - \mu[v] v \right| \leq |v| \left| \mu[u] - \mu[v] \right| + \left| \mu [u] \right| \left| u - v\right|,
\end{equation*}
it follows that
\begin{equation*}
	\begin{aligned}
	\left\| \mu[u]u - \mu[v] v \right\|_{L^{1}\left((0, T),H^1 \right)} &\lesssim T \big( \| \mu[u]\|_{L^\infty[0,T]} \| u - v\|_{L^{\infty}\left([0,T],H^1 \right)} \\ 
	 & +\| v \|_{L^{\infty}\left([0,T],H^1 \right)} \left\| \mu[u] - \mu[v] \right \|_{L^{\infty}[0,T]} \big).
	\end{aligned}
\end{equation*}
We also notice that from the embedding $H^{1}\left(\R^{d}\right) \hookrightarrow L^{2\sigma + 2}\left( \R^d \right)$ there exists $C(M,N) >0$ such that
\begin{equation}\label{eq:muHeatEst}
 \| \mu[u]\|_{L^\infty[0,T]} \lesssim \frac{1}{N^2} \big( \| u\|_{L^\infty([0,T],H^1)}^2 + \| u\|_{L^\infty([0,T],H^1)}^{2\sigma + 2}\big) \leq C,
\end{equation}
and 
\begin{equation}\label{eq:omega1}
\begin{aligned}
 \left\| \mu[u] - \mu[v] \right\|_{L^{\infty}[0,T]} \leq C \| u- v \|_{L^{\infty}\left([0,T],H^1 \right)}.
\end{aligned}
\end{equation}
For any $u_0 \in H^1(\R^d)$, let $\mathcal{F}(u_0)(u)(t) = \mathcal{F}(u)(t) $ be defined as
\begin{equation*}
	\mathcal{F}(u)(t)=e^{t\Delta} u_0+ \int_{0}^{t} e^{(t-s)\Delta} \left( g|u|^{2\sigma} u + \mu[u]u \right) \, ds .
\end{equation*}
By using the space-time estimates in Proposition \ref{prp:sptime}, the embedding $H^1(\R^d) \hookrightarrow L^r(\R^d) $ and H\"older's inequality, we obtain
\begin{equation}\label{eq:strHeat1}
	\begin{split}
	\| \mathcal F (u)\|_{L^q([0,T),L^r) \cap L^\infty([0,T],L^2)} &\lesssim \| u_0\|_{L^2} + T^{\frac{q-q^{\prime}}{q q^{\prime}}} \| u\|^{2\sigma}_{L^\infty([0,T],H^1)} \| u\|_{L^q([0,T),L^r)} \\ 
	&+ T \| \mu[u] \|_{L^\infty[0,T]} \| u\|_{L^\infty([0,T],L^2)}
	\end{split}
\end{equation}
and 
\begin{equation}\label{eq:strHeat2}
	\begin{aligned}
		\| \nabla \mathcal F (u) \|_{L^q([0,T),L^r) \cap L^\infty([0,T],L^2)} & \lesssim \| \nabla u_0\|_{L^2} + T^{\frac{q-q^{\prime}}{q q^{\prime}}} \| u\|^{2\sigma}_{L^\infty([0,T],H^1)} \| \nabla u\|_{L^q([0,T),L^r)} \\
		&+ T\| \mu[u] \|_{L^\infty[0,T]} \|\nabla u\|_{L^\infty([0,T],L^2)}.
	\end{aligned}
\end{equation}
As a consequence of \eqref{eq:strHeat1}, \eqref{eq:strHeat2} and \eqref{eq:muHeatEst}, there exist $C_1>0$ and $K(M,N) >0$ such that
\begin{equation*}
\|\mathcal F (u)\|_{L^\infty([0,T], H^1) \cap L^{q}\left((0, T), W^{1, r}\right)} \leq \\
C_1 \|u_0\|_{H^{1}}+C_1\left(T+T^{\frac{q-q^{\prime}}{q q}}\right)K M.
\end{equation*}
In the same way, by exploiting \eqref{eq:omega1}, we also have that
\begin{equation*}
	\begin{aligned}
		\|\mathcal F (u)-\mathcal F (v)&\|_{L^{\infty}\left(\left[0,T], H^1 \right)\right.} +\|\mathcal F (u)-\mathcal F (v)\|_{L^{q}\left((0, T), W^{1,r}\right)} \leq C_1\left(T+T^{\frac{q-q^{\prime}}{q q^{\prime}}}\right)K \mathrm{d}(u, v) .
	\end{aligned}
\end{equation*}
We set $ M=2 C_1\|u_0\|_{H^{1}}$
and we choose $T > 0$ small enough so that
\begin{equation}\label{eq:tempHeat1}
	C_1\left(T+T^{\frac{q-q^{\prime}}{q q}}\right)K \leq \frac{1}{2},
\end{equation} 
which is possible because
\begin{equation*}
	\frac{q-q^{\prime}}{q q^{\prime}} = \frac{2 + 2\sigma - d\sigma }{2\sigma+2}>0.
\end{equation*}
Moreover, we observe that for any $t \in [0,T]$, there exists $C_2, C_3 >0$ such that
\begin{equation*}
 \begin{aligned}
 \big\| \mathcal F (u) (t) \big\|_{L^2} & \geq \| e^{t\Delta} u_0 \|_{L^2} - \bigg\| \int_{0}^{t} e^{(t-s)\Delta} \left( g|u|^{2\sigma} u + \mu[u]u \right) \, ds \bigg\|_{L^{\infty}([0,T],L^2)} \\
 &\geq \bigg( \| u_0 \|_{L^2}^2 - 2\int_0^t \big\| \nabla e^{\tau \Delta} u_0 \big\|_{L^2}^2 \, d\tau \bigg)^\frac{1}{2} \\
 &- C_2 \big(T^{\frac{q-q^{\prime}}{q q^{\prime}}} \| u\|^{2\sigma}_{L^\infty([0,T],H^1)} \| u\|_{L^q([0,T),L^r)} + T \| \mu[u] \|_{L^\infty[0,T]} \| u\|_{L^\infty([0,T],L^2)}\big) \\
 & \geq \bigg( \| u_0 \|_{L^2}^2 - C_3 T \big\| \nabla u_0 \big\|_{L^2}^2 \bigg)^\frac{1}{2} \\
 &- C_2 \big(T^{\frac{q-q^{\prime}}{q q^{\prime}}} \| u\|^{2\sigma}_{L^\infty([0,T],H^1)} \| u\|_{L^q([0,T),L^r)} + T \| \mu[u] \|_{L^\infty[0,T]} \| u\|_{L^\infty([0,T],L^2)}\big)
 \end{aligned}
\end{equation*}
Thus, by exploiting \eqref{eq:muHeatEst}, we have that there exists $K_1(M,N) >0 $ such that
\begin{equation*}
 \inf_{t \in [0,T]} \big\| \mathcal F (u) (t) \big\|_{L^2} \geq \| u_0\|_{L^2} - K_1 (T^\frac{1}{2} + T^{\frac{q-q^{\prime}}{q q^{\prime}}} + T).
\end{equation*}
Now we set $N = \| u_0 \|_{L^2}$. By choosing $T >0$ which satisfies condition \eqref{eq:tempHeat1} and also 
\begin{equation*}
 (T^\frac{1}{2} + T^{\frac{q-q^{\prime}}{q q^{\prime}}} + T) \leq \frac{N}{2 K_1}
\end{equation*}
we obtain that $\mathcal F$ maps $E$ into itself and it is a contraction. This is enough to conclude that there exists a local in time mild solution to \eqref{eq:heat}. The uniqueness follows from the fact that $\mathcal F $ is a contraction. Moreover, we obtain the blow-up alternative \eqref{eq:bu1} because we can repeat this argument extending the solution locally in time until the $H^1$-norm of the solution diverges. 
\end{proof}

\begin{remark}\label{rem:lip}
Let us emphasize why we need the condition $\sigma \geq \frac{1}{2}$. On one hand, in the absence of the nonlocal term $\mu$, we could use the contraction principle in the set $E$ with the weaker distance 
 \begin{equation*}
	\mathrm{g}(u, v)=\|u-v\|_{L^{q}\left((0, T), L^r\right)}+\|u-v\|_{L^{\infty}\left((0, T), L^2\right)}.
\end{equation*}
Indeed, it is standard to prove that $(E,\mathrm{g})$ is a complete metric space. On the other hand, the presence of $\mu$ requires a stronger distance induced by the $L^\infty_tH^1_x$-norm, as it is clear from inequality \eqref{eq:omega1}. This implies that we have to use the distance $\mathrm{d}$ instead of $\mathrm{g}$. But for $\sigma < \frac{1}{2}$, this is not possible because the power-type nonlinearity is not locally Lipschitz continuous in Sobolev spaces (namely, inequality \eqref{eq:omega2} fails).
\end{remark}

\subsection{Global well-posedness}\label{ss:gwp}

From Proposition \ref{prop:upgrade} we obtain that for any $u_0\in H^1_0(\Omega)$ and $t\in [0,T_{max}(u_0)]$, the energy of the corresponding solution satisfies
	\begin{equation*}
		E[u(t)] = E[u_0] - \int_0^t \int |\partial_s u(s)|^2 \,dx \, ds,
	\end{equation*}
	and thus it is non-increasing in time. 	The global existence follows from a classical argument which combines the blow-up alternative \eqref{eq:bu1} and the bound on the energy.
 
\begin{theorem}
		Let $g \leq 0$, $\sigma < \frac{2}{(d-2)^+}$ or $g > 0$, $\sigma < \frac{2}{d}$. Let $\Omega$ be a regular bounded domain or $\Omega = \R^d$. Then for any
		$ u_0 \in H^1_0(\Omega)$ there exists a unique global solution $u \in C([0,\infty), H^1_0(\Omega))$ to \eqref{eq:heat}.
\end{theorem} 

\begin{proof}
		Let $u_0 \in H^1_0(\Omega)$ and let $u \in C([0,T_{max}), H^1_0(\Omega))$ be the corresponding local solution to \eqref{eq:heat} given by Theorems \ref{thm:proof2}, \ref{thm:proof1} or \ref{thm:lwp}. Then the $L^2$-norm of the solution is constant and the energy is non-increasing in time. Thus for $g \leq 0$, $0 < \sigma < \frac{2}{(d-2)^+}$, we have
		\begin{equation*}
			\| \nabla u(t) \|_{L^2}^2 \leq 2 E[u(t)] \leq 2E[u_0],
		\end{equation*} 
		while for $g > 0$, $\sigma < \frac{2}{d}$, we use the Gagliardo-Nirenberg's inequality 
		\begin{equation*}
			\|\nabla u(t)\|_{L^2}^2 = E[u(t)] + \frac{g}{\sigma + 1} \|u(t)\|_{L^{2\sigma + 2}}^{2\sigma + 2} \lesssim E[u_0] + \|\nabla u(t)\|^{d\sigma}_{L^2},
		\end{equation*}
	and $d\sigma < 2$ to conclude that
		\begin{equation*}
			\| \nabla u(t) \|_{L^2} \lesssim \| \nabla u_0\|_{L^2}.
		\end{equation*}
		In particular, in both cases, the $H^1$-norm of the solution is uniformly bounded in time. From the blow-up alternative \eqref{eq:bu1} we conclude that the solution is global. 
\end{proof}
	
In what follows, we will provide an alternative sufficient condition for global existence in bounded domains, as stated in Theorem \ref{thm:pot1}. The proof is based on the potential well method which was introduced in Section \ref{sec:potwel}. We start with the following lemma where we use the notations introduced in Section \ref{sec:potwel}. 
	
\begin{lem}\label{thm:potwel}
 Let $g > 0$ and $0 < \sigma < \frac{2}{(d-2)^+}$. If $u_0 \in \mathcal W$ (respectively $u_0 \in \mathcal Z$), then the corresponding solution $u\in C([0,T_{max}(u_0)), H^1_0(\Omega))$ to \eqref{eq:heat} is so that $u(t) \in \mathcal W$ (respectively $u(t) \in \mathcal Z$) for all $t \in [0,T_{max}(u_0))$. 
\end{lem}
	
\begin{proof}
	Fix $g = 1$ without losing generality and suppose that $u_0 \in \mathcal{W}$. Let $T = T_{max}(u_0)$ be the maximal time of existence of the corresponding solution to \eqref{eq:heat}. Equation \eqref{eq:derEn} implies that $E[u(t)] \leq E[u_0] < p$, for all $t \in [0,T)$. 
\newline
 Since $I[u_0] >0$, it follows by the continuity of the flow associated with \eqref{eq:heat} that $I[u(t)] >0$ for all $t \in [0,T)$. Indeed, by contradiction, suppose that there exists a time $t_0 >0$ such that $ I[u(t_0)] =0$ and $E[u(t_0)] < p$. Notice that this contradicts the definition of $p$. Hence $u(t) \in \mathcal W$. 
 \newline
 The same argument applies for the case when $u_0 \in \mathcal Z$.
\end{proof}

Lemma \ref{thm:potwel} implies that the sets $\mathcal W$ and $\mathcal Z$ are invariant under the flow of equation \eqref{eq:heat}. As a consequence, we obtain the following.

\begin{theorem}
 Let $g > 0$ and $0 < \sigma < \frac{2}{(d-2)^+}$. If $u_0 \in \mathcal W$, then the corresponding solution $u\in C([0,T_{max}(u_0)), H^1_0(\Omega))$ to \eqref{eq:heat} is global in time, that is $T_{max}(u_0) = \infty$.
\end{theorem}

\begin{proof}
 We fix $g = 1$ without losing generality. The condition $u_0 \in \mathcal W$ and Lemma \ref{thm:potwel} imply that $u(t) \in \mathcal{W}$ for every $t \in [0,T_{max}(u_0))$ and consequently that
 \begin{equation*}
	 \| \nabla u(t) \|_{L^2}^2 > \| u(t) \|_{L^{2\sigma + 2}}^{2\sigma + 2}.
 \end{equation*}
 Equation \eqref{eq:derEn} then implies that 
 \begin{equation*}
 	\left( \frac{1}{2} - \frac{1}{2\sigma + 2}\right) \| \nabla u(t) \|_{L^2}^2 \leq E[u(t)] \leq E[u_0] < p.
 \end{equation*}
 Hence the $H^1$-norm of $u$ is uniformly bounded by $\frac{(2\sigma + 2) p }{\sigma}$. From the blow-up alternative \eqref{eq:bu1}, we conclude that $T_{max}(u_0) = \infty$. 
\end{proof}

\begin{remark}
Proposition \ref{prp:minim} provides a sufficient smallness condition on the initial datum for global existence. Indeed, if $u_0 \in H^1_0(\Omega)$ is so that $\| \nabla u_0 \|_{L^2}^2 \leq 2p$, then $u_0 \in \mathcal{ W}$. 
\end{remark}

\begin{remark}
	Notice that when $\Omega$ is a bounded domain, the set $\mathcal W$ is not empty because $\| \nabla u\|_{L^2} < \sqrt{2p}$ implies $u \in \mathcal{W}$ (see Proposition \ref{prp:minim}), while the set $\mathcal{Z}$ is not empty since if $E[u] \leq 0$, then from \eqref{eq:itoE}, $I[u] < 0$ and $u \in \mathcal{Z}$.
	 \newline
 On the other hand, when $\Omega = \R^d$, we have $p = 0$ and $\mathcal W$ is empty. Indeed, take any $u \neq 0$ such that $I[u] = 0$. Notice that $E[u] > 0$. 
	Let $v(x) = \lambda^\frac{1}{\sigma} u(\lambda x)$. Then 
	\begin{equation*}
	 I[v] = \lambda^{2 + \frac{2}{\sigma} - d} I[u] = 0,
	\end{equation*}
	while 
	\begin{equation*}
	 E[v] = \lambda^{2 + \frac{2}{\sigma} - d} E[u] > 0.
	\end{equation*}
	Thus, for $\lambda \to 0$, we obtain that $E[v] \to 0$. As a consequence, from \eqref{eq:itoE}, the set $\mathcal{W}$ is empty when $\Omega = \R^d$	and Theorem \ref{thm:pot1} does not provide any additional sufficient conditions for the global existence in $H^1(\R^d)$. 
	\newline
	Finally, if $u\in H^1(\R^d)$ is such that $E[u] < 0$, then $u \in \mathcal{Z}$ so $\mathcal{Z}$ is not empty. Note that \eqref{eq:poh1} and \eqref{eq:poh2} imply that a stationary state $Q$, solution to \eqref{eq:st_st} in $\R^d$, satisfies
	\begin{equation}\label{eq:enQ}
	 E[Q] = \frac{d\sigma -2 }{2d\sigma} \| \nabla Q\|_{L^2}^2,
	\end{equation}
	 and thus, when $\sigma < \frac{2}d$, $Q$ belongs to $\mathcal Z$.
	\end{remark} 
	
We proceed by providing additional sufficient conditions for global existence in the whole space $\R^d$, as stated in Theorem \ref{thm:potK}. We use the same notations introduced in Section \ref{sec:potK}. We start by proving that the set $\mathcal K$ is invariant under the flow of equation \eqref{eq:heat}.

\begin{lem}
Let $\frac{2}{d} < \sigma < \frac{2}{d-2}$ and $g \geq 0$.	If $u_0 \in \mathcal K$, then the corresponding solution $u \in C([0,T_{max}(u_0)), H^1(\R^d))$ to \eqref{eq:heat} is so that $u(t) \in \mathcal K$ for all $t \in [0,T_{max}(u_0))$.
\end{lem}
 
\begin{proof}
 We suppose that $g = 1$ without losing generality. By Gagliardo-Nirenberg's inequality \eqref{eq:GN} it follows that
 	\begin{equation}\label{eq:k1}
 		E[u] \| u \|_{L^2}^{2\alpha} \geq \frac{1}{2} \| \nabla u\|_{L^2}^2 \| u \|_{L^2}^{2\alpha} - \frac{C_{GN}}{2\sigma + 2} \| \nabla u\|_{L^2}^{d\sigma} \| u\|_{L^2}^{d\sigma \alpha} = f(\| \nabla u\|_2 \| u \|_{L^2}^{\alpha}),
 	\end{equation}
 where
 	\begin{equation*}
 		f(x) = \frac{1}{2} x^2 - \frac{C_{GN}}{2\sigma + 2}x^{d\sigma}.
 	\end{equation*}
 By taking the derivative of $f$, we notice that it has two critical points
	\begin{equation*}
 	f'(0) = 0, \ \mbox{ and } f'(x_1) = 0,
	\end{equation*}
 and
	\begin{equation*}
		x_1 = \left( \frac{2\sigma + 2}{d\sigma C_{GN}} \right)^{\frac{1}{d\sigma - 2}},
	\end{equation*}
 is a local maximum. From \eqref{eq:cGN} we can write
 \begin{equation*}
 x_1 = \| \nabla Q\|_{L^2} \| Q\|_{L^2}^{\alpha}.
 \end{equation*}
 Moreover, since $\eqref{eq:GN}$ is an equality for $Q$, we have
 \begin{equation*}
f(x_1) = f(\| \nabla Q\|_{L^2} \| Q \|_{L^2}^{\alpha}) = E[Q]\| Q\|_{L^2}^{2\alpha}. 
 \end{equation*}
From the condition $E[u_0] \| u_0\|_{L^2}^{2\alpha} < E[Q] \| Q\|_{L^2}^{2\alpha} = f(x_1) $, \eqref{eq:k1} and \eqref{eq:derEn} we obtain that for any $t \in [0,T_{max}(u_0))$, the following inequality is true 
	\begin{equation}\label{eq:contrK}
		f(\| \nabla u(t)\|_{L^2} \| u(t) \|_{L^2}^{\alpha}) \leq E[u(t)]\| u(t) \|_{L^2}^{2\alpha} \leq E[u_0] \| u_0 \|_{L^2}^{2\alpha} < f(x_1).
	\end{equation}
Let $u_0\in H^1(\R^d)$ be such that $ \| \nabla u_0\|_{L^2} \| u_0 \|_{L^2}^{\alpha} < \| \nabla Q\|_{L^2} \| Q \|_{L^2}^{\alpha}$. Suppose by contradiction that there exists a time $t_0 > 0$ such that $\| \nabla u(t_0)\|_{L^2} \| u(t_0) \|_{L^2}^{\alpha} = \| \nabla Q\|_{L^2} \| Q \|_{L^2}^{\alpha}$. Then we would have $f(\| \nabla u(t_0)\|_{L^2} \| u(t_0) \|_{L^2}^{\alpha}) = f(x_1)$ which is in contradiction with \eqref{eq:contrK}. Thus we can conclude that the set $\mathcal K$ is invariant under the flow of equation \eqref{eq:heat}.
\end{proof}
 
\begin{theorem}
 		Let $\frac{2}{d} < \sigma < \frac{2}{d-2}$ and $g \geq 0$.	If $u_0 \in \mathcal K$, then the corresponding solution $u \in C([0,T_{max}(u_0)) ), H^1(\R^d))$ to \eqref{eq:heat} is global in time. 
 \end{theorem}
 
 \begin{proof}
 	Since $\mathcal K$ is invariant under the flow of equation \eqref{eq:heat}, it follows from the conservation of the $L^2$-norm that for any $t \in [0,T_{max}(u_0))$, 
 	\begin{equation*}
 		\| \nabla u(t) \|_{L^2}^2 < 	\| \nabla Q\|_{L^2}^2 \frac{\| Q\|^{2\alpha}_{L^2}}{\|u_0\|_{L^2}^{2\alpha}}
 	\end{equation*}
 and thus the $H^1$-norm of the solution is uniformly bounded in time. The blow-up alternative \eqref{eq:bu1} implies that the solution is global in time. 
 \end{proof}

\section{Asymptotic Behavior}\label{sec:asy}

In the following proposition, $\Omega$ can be either a regular bounded domain or the whole $\R^d$.
\begin{prop} \label{prp:weak}
	Let $g \leq 0$, $\sigma < \frac{2}{(d-2)^+}$ or if $g > 0$, $\sigma < \frac{2}{d}$. Let $u_0 \in H^1_0(\Omega)$ and let $u \in C([0,\infty), H^1_0(\Omega))$ be the solution to \eqref{eq:heat}. Then there exists a sequence $\{ t_n \}_{n \in \N}$, $t_n \rightarrow \infty$ as $n \rightarrow \infty$, such that 
	\begin{equation*}
			u(t_n) \rightharpoonup u_\infty \ \mbox{ in } \ H^1_0(\Omega), \quad \quad \mu[u(t_n)] \rightarrow \mu[u_\infty] \ \mbox{ as } \ n\rightarrow \infty
	\end{equation*}
	where $u_\infty$ solves the stationary equation
	\begin{equation*}
		\begin{cases}
			\Delta u_\infty + g|u_\infty|^{2\sigma} u_\infty + \mu[u_\infty] u_\infty = 0,\\
			 \| u_\infty\|_{L^2} \leq \| u_0\|_{L^2}.
		\end{cases}
	\end{equation*}
\end{prop}
	
\begin{proof}
		Since $\partial_t u \in L^2([0,\infty), L^2(\Omega))$, $\sup_{t >0} \left|\mu[u(t)]\right|\leq C$, and $\| u\|_{L^\infty([0,\infty),H^1)} \leq C$, then there exists a sequence $\{ t_n \}_{n \in N}$, $t_n \rightarrow \infty$ as $n \rightarrow \infty$ such that
		\begin{equation*}
			\begin{cases}
				u(t_n) \rightharpoonup u_\infty \ \mbox{ in } \ H^1_0(\Omega), \\
				 \partial_t u(t_n) \rightarrow 0 \ \mbox{in} \ L^2(\Omega), \\
				 \mu[u(t_n)] \rightarrow \mu_\infty,
			\end{cases}
		\end{equation*}
		and 
	\begin{equation*}
	 \partial_t u(t_n) - \Delta u(t_n) - |u(t_n)|^{2\sigma}u(t_n) - \mu[u(t_n)] u(t_n) \rightharpoonup - \Delta u_\infty - |u_\infty|^{2\sigma}u_\infty - \mu_\infty u_\infty,
	\end{equation*}
		where the last convergence is in a weak $H^{-1}(\Omega)$ sense. In particular, we get that
		\begin{equation*}
			\Delta u_\infty + g |u_\infty|^{2\sigma} u_\infty + \mu_\infty u_\infty = 0.
		\end{equation*}
		By taking the scalar product of this equation with $u_\infty$, we obtain that $\mu_\infty = \mu[u_\infty]$. 
	\end{proof}
	The result of the Proposition \ref{prp:weak} can be improved in a bounded domain $\Omega \subset \R^d$ due to the Rellich-Kondrachov theorem. Indeed:
	\begin{cor}\label{cor:sub}
		Under the hypothesis of Proposition \ref{prp:weak}, if $\Omega \subset \R^d$ is a regular, bounded domain, then
		\begin{equation*}
			u(t_n) \rightarrow u_\infty \ \mbox{ in } \ H^1_0(\Omega) \ \mbox{ as } \ n\rightarrow \infty
		\end{equation*}
		and $u_\infty$ solves the stationary equation
		\begin{equation*}
			\begin{cases}
				&	\Delta u_\infty + \mu[u_\infty] u_\infty + g|u_\infty|^{2\sigma} u_\infty = 0,\\
				& \| u_\infty\|_{L^2} = \| u_0\|_{L^2}.
			\end{cases}
		\end{equation*}
	\end{cor}
	\begin{proof}
		The compact embedding $H^1_0(\Omega) \hookrightarrow L^2(\Omega) \cap L^{2\sigma + 2}(\Omega)$ and the weak convergence $ u(t_n) \rightharpoonup u_\infty$ in $H^1$ imply $\| u_\infty\|_{L^2} = \| u_0\|_{L^2}$ and $\| u(t_n) \|_{L^{2\sigma +2}}^{2\sigma +2} \rightarrow \| u_\infty \|_{L^{2\sigma +2}}^{2\sigma +2}$. Combining this convergence with 
		\begin{equation*}
			\mu[u(t_n)] = \frac{\| \nabla u(t_n) \|^2_{L^2} - g \| u(t_n)\|^{2\sigma + 2}_{L^{2\sigma + 2}}}{{\|u_0\|_{L^2}^2}} \rightarrow \mu[u_\infty] = \frac{\| \nabla u_\infty \|^2_{L^2} - g \| u_\infty\|^{2\sigma + 2}_{L^{2\sigma + 2}}}{{\|u_0\|_{L^2}^2}},
		\end{equation*}
		we obtain that $ u(t_n) \rightarrow u_\infty$ in $H^1_0(\Omega)$ as $n\to \infty$.
	\end{proof}

The result of the Corollary \ref{cor:sub} can be further improved by using the Maximum Principle. In this case, we shall work in a bounded domain $\Omega$ where there exists a unique positive solution to the stationary equation \eqref{eq:st_st} which is also the unique minimizer of the problem \eqref{eq:minpr}.
For example, we can take $\Omega = B(0,R)$ and $g > 0$. Then we can prove that if $u_0 \geq 0$, the corresponding solution converges strongly to the ground state.

\begin{theorem}
		Let $g \leq 0$, $\sigma < \frac{2}{(d-2)^+}$ or $g > 0$, $\sigma < \frac{2}{d}$. Let $u_0 \in H^1_0(\Omega)$, $u_0 \geq 0$, $u_0 \centernot{\equiv} 0$ and let $u \in C([0,\infty), H^1_0(\Omega))$ be the solution to \eqref{eq:heat}. Let $\Omega \subset \R^d$ be a bounded domain such that there exists a unique positive solution $Q \in H^1_0(\Omega)$ to equation \eqref{eq:st_st} which is also the unique minimizer of the problem \eqref{eq:minpr}. Then 
		\begin{equation*}
			u(t) \rightarrow Q \ \mbox{ in } \ H^1_0(\Omega)\ \mbox{ as } \ t\rightarrow \infty.
		\end{equation*}
\end{theorem}
	
\begin{proof}
		Let $\{t_n\}_{n\in \N}$, $t_n \rightarrow \infty$ as $n\to \infty$, be such that $u(t_n) \rightarrow u_\infty$ in $H^1_0(\Omega)$ where $u_\infty$ satisfies \eqref{eq:st_st}. By the Maximum Principle (see \cite[Ch. 2, Sec. 4]{Fr2013}), $u(t_n) \geq 0$ which implies $u_\infty \geq 0$ and $u_\infty \centernot{\equiv} 0$ since $\| u_\infty\|_{L^2} = \| u_0\|_{L^2}>0$. From the uniqueness of the positive solution to \eqref{eq:st_st}, we conclude that $u_\infty = Q$ where $Q$ is the unique minimizer of the energy under the constraint $\|Q\|_{L^2} = \| u_0\|_{L^2}$. Equation \eqref{eq:derEn} implies that the energy is a continuous, decreasing function of time. In particular, there exists the limit
		\begin{equation}\label{eq:grt}
			E[u(t)] \searrow E_{\infty} \geq E[Q],
		\end{equation}
		as $t \to \infty$. 	This limit $E_\infty$ is the energy of the ground state because 
		\begin{equation*}
		 E_\infty = \lim_{t \rightarrow \infty} E[u(t)] = \lim_{n \rightarrow \infty} E[u(t_n)] = E[Q].
		\end{equation*}
		Now we show that $u(t) \rightarrow Q$ in $H^1_0(\Omega)$. Suppose, by contradiction, that this convergence is not true. In particular, there exists a time sequence $\{t_k\}_{k\in \N}$ such that $\| u(t_k)\|_{L^2} = \| u_0\|_{L^2},$ $E[u(t_k)] \rightarrow E[Q]$ and there exists $\varepsilon > 0$ so that 
		\begin{equation*}
		 \| u(t_k) - Q\|_{H^1} \geq \varepsilon \ \mbox{for all} \ k \in \N.
		\end{equation*}
		 Since $\sup_{k\in\N}\|u(t_k)\|_{H^1} \leq C < \infty$, there exists a subsequence, still denoted by $\{t_k\}_{k\in \N}$ and a profile $u \in H^1_0(\Omega)$ so that $u(t_k) \rightharpoonup u$ in $H^1_0(\Omega)$. By the compact embedding $H^1_0(\Omega) \hookrightarrow L^{2\sigma +2}(\Omega)$, we obtain $\|u(t_k)\|_{L^{2\sigma + 2}} \rightarrow \|u\|_{L^{2\sigma + 2}}$. By the lower semi-continuity, it follows that
		\begin{equation*}
			E[u] \leq \liminf_k E[u(t_k)] = E[Q],
		\end{equation*}
		which implies that $E[u] = E[Q]$. In particular, $\| \nabla u(t_k)\|_{L^2} \rightarrow \|\nabla Q\|_{L^2}$, and so $u(t_k) \rightarrow Q$ in $H^1_0(\Omega)$. This is a contradiction with $\| u(t_k) - Q\|_{H^1} \geq \varepsilon$. 
\end{proof}
	
\begin{remark}
	We would like to notice that in general, there exist initial data not converging to the ground state. Indeed if $g = 1$, $\sigma < \frac{2}{d}$ and $u_0 \in \mathcal Z$, where $\mathcal Z$ is defined in \eqref{eq:zdef}, then $u_\infty$ where $u_\infty \in \mathcal{Z}$ is a stationary state with $\mu[u_\infty] < 0$. Similarly, if $u_0 \in \mathcal W$ where $\mathcal W$ is defined in \eqref{eq:wdef} then $u_\infty \in \mathcal W$ and $\mu[u_\infty] > 0$. Thus, at least in one of the above cases, the solution does not converge to the ground state. 
\end{remark}	

 Finally, we show that in the whole space $\R^d$ and when $\sigma \geq \frac{2}{d}$ there exists a set of initial data whose solutions satisfy the grow-up condition. Specifically, if $E[u_0] < 0$, then there exists a sequence $\{t_k\}_{k\in \N}$, $t_k \rightarrow T_{max}$ such that 
 \begin{equation*}
 \lim_{k \to \infty} \| \nabla u(t_k)\|_{L^2} = \infty.
 \end{equation*}
 
 \begin{theorem}
	 Let $g = 1 $, $\sigma \geq \frac{2}{d}$. Let $u_0 \in H^1(\R^d)$ be such that $E[u_0] < 0$. Let $T_{max} (u_0) > 0$ be the maximum time of existence of the corresponding strong solution to \eqref{eq:heat}. Then either $T_{max}(u_0) = \infty$ and
	 \begin{equation*}
	 \limsup_{t \rightarrow \infty} \| \nabla u(t)\|_{L^2} = \infty,
	 \end{equation*}
	 or $T_{max}(u_0) < \infty$.
	\end{theorem}
	\begin{proof}
	 Suppose that $T_{max}(u_0) = \infty$ and that there exists a constant $C>0$ such that 
	 \begin{equation*}
	 \sup_{t \geq 0} \| \nabla u\|_{L^2} \leq C. 
	 \end{equation*}
	Gagliardo-Nirenberg's inequality and the conservation of the $L^2$-norm imply that
	\begin{equation*}
	 E[u(t)] \geq - \frac{1}{2\sigma + 2} \| u(t)\|_{L^{2\sigma + 2}}^{2\sigma + 2} \gtrsim - \| \nabla u(t)\|_{L^2}^{d\sigma} \geq -C^{d\sigma}.
	\end{equation*}
	Thus the energy is bounded from below. Since $E[u(t)]$ is a continuous, decreasing function, there exists the limit 
	\begin{equation*}
	 \lim_{t\to \infty} E[u(t)]= E_\infty > -\infty,
	\end{equation*}
	and $\partial_t u \in L^2([0,\infty),L^2(\R^d))$ because
	\begin{equation*}
	 \int_0^\infty \| \partial_t u(s)\|_{L^2}^2 ds = E[u_0] - E_\infty < \infty.
	\end{equation*}
	In the same way, we can see that $\mu[u(t)]$ is bounded from below 
	\begin{equation*}
	 \mu[u(t)] =\frac{\| \nabla u(t)\|_{L^2} - \| u(t)\|_{L^{2\sigma + 2}}^{2\sigma + 2}}{\|u_0\|_{L^2}^2} \gtrsim - \| u(t)\|_{L^{2\sigma + 2}}^{2\sigma + 2} \geq -C^{d\sigma},
	\end{equation*}
	and from above. 
	Then there exists a sequence $\{ t_n \}_{n \in \N}$, $t_n \rightarrow \infty$ as $n \rightarrow \infty$ such that
	\begin{equation*}
			\begin{cases}
				 u(t_n) \rightharpoonup u_\infty \ \mbox{ in } \ H^1(\R^d), \\
				 \partial_t u(t_n) \rightarrow 0 \ \mbox{in} \ L^2(\R^d),\\
				 \mu[u(t_n)] \rightarrow \mu_\infty, 
			\end{cases}
	\end{equation*}
	and 
	\begin{equation*}
	 \partial_t u(t_n) - \Delta u(t_n) - |u(t_n)|^{2\sigma}u(t_n) - \mu[u(t_n)] u(t_n) \rightharpoonup - \Delta u_\infty - |u_\infty|^{2\sigma}u_\infty - \mu_\infty u_\infty,
	\end{equation*}
		where the last convergence is in a weak $H^{-1}(\Omega)$ sense. In particular, we get that
		\begin{equation*}
			\Delta u_\infty + |u_\infty|^{2\sigma} u_\infty + \mu_\infty u_\infty = 0,
		\end{equation*}
		and so $u_\infty$ is a stationary state. From \eqref{eq:enQ}
		\begin{equation*}
		 E[u_\infty] = \frac{d\sigma -2 }{2d\sigma} \| \nabla u_\infty\|_{L^2}^2 \geq 0.
		\end{equation*}
		On the other hand, by the weak lower semicontinuity, we have
		\begin{equation*}
		 E[u_\infty] \leq \lim_n E[u(t_n)] \leq E[u_0] <0
		\end{equation*}
		which is a contradiction with the former inequality. Hence $\| \nabla u(t)\|_{L^2}$ cannot be uniformly bounded in time. 
	\end{proof} 
 \section*{acknowledgements}
 We sincerely thank the anonymous referee for invaluable suggestions and for pointing out a flaw in a previous version of the manuscript that led us to prove the LWP result in Subsection $3.1.$ as it is presented now. \newline
%We extend our sincere gratitude to the anonymous referee for the invaluable suggestions they provided.\newline
The second author was partly supported by Istituto Nazionale di Alta Matem- atica (GNAMPA Research Projects) and by the MIUR Excellence Department Project awarded to the Department of Mathematics, University of Rome Tor Vergata, CUP E83C18000100006. The third authors is partially supported by Istituto Nazionale di Alta Matematica (GNAMPA group).
\section*{Conflict of interests declaration}
The authors declare no conflict of interests.
\section*{Data availability statement}
Data sharing not applicable to this article as no datasets were generated or analysed during the current study.
 
	\bibliographystyle{plain}
	\bibliography{heat_eq}
 \end{document}